\documentclass[12pt,reqno]{amsart}
\usepackage{amssymb}
\usepackage{amsfonts}
\usepackage{indentfirst}
\usepackage{amsopn}
\usepackage{amsmath}
\usepackage{amsthm}
\usepackage{amscd}
\usepackage{graphicx}
\usepackage{epstopdf}
\usepackage{color}

\makeatletter\@addtoreset {equation}{section}\makeatother

\newtheorem{theorem}{Theorem}

\newtheorem{proposition}{Proposition}
\newtheorem{lemma}{Lemma}
\newtheorem{remark}{Remark}

\begin{document}

\title[Linear instability of breathers for the NLS equation]{Linear instability of breathers for the focusing nonlinear Schr\"{o}dinger equation}

\author{Mariana Haragus, Dmitry E. Pelinovsky}
\address[M. Haragus]{FEMTO-ST institute, Univ. Bourgogne Franche-Comt\' e, CNRS, 15b avenue des Montboucons, 25030 Besan\c con cedex, France}
\email{mharagus@univ-fcomte.fr}
\address[D.E. Pelinovsky]{Department of Mathematics, McMaster University, Hamilton, Ontario, Canada, L8S 4K1}
\email{dmpeli@math.mcmaster.ca}

\date{\today}

\begin{abstract}
Relying upon tools from the theory of integrable systems, we discuss the linear instability of the Kuznetsov--Ma breathers and the Akhmediev breathers of the focusing nonlinear Schr\"{o}dinger equation. We use the Darboux transformation to construct simultaneously the breathers and the exact solutions of the Lax system associated with the breathers. We obtain a full description of the Lax spectra for the two breathers, including multiplicities of eigenvalues. Solutions of the linearized NLS equations are then obtained from the eigenfunctions and generalized eigenfunctions of the Lax system. While we do not attempt to prove completeness of eigenfunctions, we aim to determine the entire set of solutions of the linearized NLS equations generated by the Lax system in appropriate function spaces. 
\end{abstract}

\maketitle

\section{Introduction}

The focusing nonlinear Schr\"{o}dinger (NLS) equation in the space of 
one dimension is a fundamentally important model which brings together  nonlinearity and dispersion of modulated waves in many physical systems \cite{NLS2,NLS1}. It has been used as the main testbed for rogue 
waves in fluids and optics \cite{Rogue1,Rogue2}, where the rogue waves 
appear from nowhere and disappear without any trace. One of the important 
properties of the focusing cubic NLS equation is its integrability, 
which allows to construct the basic solutions for the rogue waves 
in a closed analytical form. Although these solutions have been constructed 
long ago in the works of Akhmediev {\em et al.} \cite{Nail1}, Kuznetsov \cite{Kuznetsov}, Ma \cite{Ma}, Peregrine \cite{Peregrine}, 
and Tajiri \& Watanabe \cite{TW}
they have been studied a lot in the past few years in physics 
literature \cite{Kibler1,Kibler2}.

To explain the current state of art in the mathematical studies of these 
breather solutions, we set up the stage and take the NLS equation in the following dimensionless form:
\begin{equation}
i \psi_t + \frac{1}{2} \psi_{xx} + |\psi|^2 \psi = 0,
\label{nls}
\end{equation}
where the unknown $\psi$ is a complex-valued function depending on time $t\in\mathbb R$ and space $x\in\mathbb R$.
The NLS equation (\ref{nls}) is invariant under the scaling transformation 
\begin{equation*}
\mbox{\rm if} \; \psi(x,t) \;\; \mbox{\rm is a solution, so is} \; 
c \psi(c x,c^2 t),
\;\; \mbox{\rm for every} \;\; c \in \mathbb{R},
\end{equation*}
and under translations in $x$ and $t$.
Up to these symmetries, the NLS equation (\ref{nls}) admits the following exact solutions on the background of the constant-amplitude wave $\psi(x,t) = e^{it}$:
\begin{itemize}
        \item Akhmediev breather (AB)
\begin{equation}
\label{AB}
\psi(x,t) = \left[ - 1 + \frac{2(1-\lambda^2) \cosh(\lambda k t) + i \lambda k \sinh(\lambda k t)}{\cosh(\lambda k t) - \lambda \cos(kx)} \right] e^{it},
\end{equation}
where $k = 2 \sqrt{1-\lambda^2}$, and $\lambda \in (0,1)$ is the only free parameter.
	\item Kuznetsov--Ma breather (KMB)
\begin{equation}
\label{KM}
\psi(x,t) = \left[ - 1 + \frac{2(\lambda^2-1) \cos(\lambda \beta t) + i \lambda \beta \sin(\lambda \beta t)}{\lambda \cosh(\beta x) - \cos(\lambda \beta t)} \right] e^{it},
\end{equation}
where $\beta = 2 \sqrt{\lambda^2-1}$, and $\lambda \in (1,\infty)$ is the only free parameter.
        \item Peregrine's rogue wave (PRW)
\begin{equation}
\label{P}
\psi(x,t) = \left[ - 1 + \frac{4(1+2it)}{1 + 4(x^2 + t^2)} \right] e^{it}.
\end{equation}
\end{itemize}
Note that PRW can be obtained in the limit $\lambda \to 1$ from either AB or KMB. Also note the formal transformation $k = i \beta$ between AB and KMB. 
The main goal of this work is to study the linear instability of AB and KMB.

Stability of breathers is a challeging question that has been extensively studied in the mathematics literature. A major difficulty comes from the nontrivial dependence of the breathers on time $t$. Therefore, many of the analytical methods developed for stability of stationary or traveling waves in nonlinear partial differential equations do not apply to breathers. For instance, spectral methods are out of reach for AB and PRW which are localized in time $t$. Since KMB is periodic in time $t$, Floquet theory can be used, at least formally, to compute stable and unstable modes of KMB. This has been done numerically in \cite{Cuevas} after KMB were truncated on a spatially periodic domain in $x$. Further studies of KMB in discrete setting of the NLS equation can be found in \cite{Kevr-Salerno}.

Very recently, the authors of \cite{Z} set up a basis for a rigorous investigation of stability of breathers which are periodic in time $t$ and localized in space $x$. Using tools from the theory of semigroups and Fredholm operators, they analyzed properties of the monodromy operator for the linearization of the cubic-quintic complex Ginzburg--Landau equation about such solutions, and computed its essential spectrum. These results being obtained in a dissipative setting do not directly apply to KMB due to the Hamiltonian nature of the NLS equation.

Most of the existing instability results for breathers of the NLS equation strongly rely upon the integrability properties of the NLS equation.
Instability of AB was concluded by using the variational characterization 
of breathers in the energy space in \cite{AFM-2}. Perturbations 
to the AB were considered in the periodic space $H^s_{\rm per}$ for $s > 1/2$. Similar techniques were applied to KMB and PRW in \cite{AFM} (see also the review in \cite{Alejo}). It was shown that both KMB and PRW are unstable with respect 
to perturbations in $H^s(\mathbb{R})$ for $s > 1/2$.

Evolution of KMB and PRW under perturbations 
was studied in \cite{Garn} and \cite{BK}, where inverse scattering transform 
was applied to the NLS equation in the class of functions decaying 
to the nonzero boundary conditions.
Instability of PRW was visualized numerically in \cite{KH} by using time-dependent simulations of the NLS equation. 
Linear instability of PRW was also studied numerically in \cite{CalSch2019}. 
By using perturbation theory for embedded eigenvalues of the Lax system, 
it was shown in \cite{KPR} that the perturbed PRW is transformed 
to either KMB or two counter-propagating breathers, the latter solutions 
were later constructed explicitly in \cite{ZakGelash}.

Our approach to linear instability is closely related to the recent works
\cite{Bil2, CalSch2012,GS-2021}, where solutions of the linearized NLS equation are constructed from solutions of the associated Lax system. Eigenfunctions of the Lax system related to the Lax spectrum provide solutions of the linearized NLS equation relevant for the linear instability of breathers. The completeness of the resulting solution set is a particularly challenging question. In the class of spatially localized functions, it was shown in \cite[Section 3.4]{Bil2} how  to obtain a complete set of solutions of the linearized NLS equation at PRW. Stability of AB 
under periodic perturbations of the same period was stated in \cite{CalSch2012} without the proof of completeness. It was recently discovered in \cite{GS-2021} 
that the set of eigenfunctions constructed in \cite{CalSch2012} is incomplete 
and two unstable modes exist for AB. The spatially periodic unstable modes 
for AB were constructed in \cite{GS-2021} by taking a suitable combination of unbounded solutions of the linearized NLS equation.

The purpose of this paper is twofold. Firstly, we give a full description of the Lax spectra for AB and KMB, including algebraic multiplicities of eigenvalues. Secondly, we obtain all solutions of the linearized NLS equations at AB and KMB generated by eigenfunctions and generalized eigenfunctions of the Lax systems. These solutions are spatially periodic for AB and spatially localized for KMB. The completeness question is outside the scope of this paper and will be the subject of subsequent studies. 

Similar to \cite{Bil2,CalSch2012,GS-2021}, we use the Darboux transformation to obtain AB and KMB from the constant-amplitude wave and then to precisely determine the Lax spectra at AB and KMB from the Lax spectrum at the constant-amplitude wave. For AB we focus on solutions of the linearized NLS equation with the first three spatially periodic Fourier modes, whereas for KMB we focus on spatially localized solutions.

Aiming for a presentation accessible to readers who are not expert in integrable systems, we review some properties of the Lax system and the Darboux transformation in Section~\ref{s prelim}. In Section~\ref{s constant} we consider the constant-amplitude wave. We compute the Lax spectrum and establish the explicit relation between the solutions of the linearized NLS equation obtained by a standard Fourier analysis and the ones generated by  the Lax system. We focus on spatially periodic and spatially localized solutions. Then using the Darboux transformation, we determine the Lax spectra and the resulting solutions of the linearized NLS equations for AB in Section~\ref{s AB} and for KMB in Section~\ref{s KMB}. The paper is concluded 
at Section \ref{sec-conclusion} with a discussion of further directions. 

\medskip

\noindent
{\bf Acknowledgments:}
M. Haragus was partially supported by the project Optimal (ANR-20-CE30-0004) and the EUR EIPHI program (ANR-17-EURE-0002). D. E. Pelinovsky was partially supported by the National Natural Science Foundation of China (No. 11971103).

\section{Preliminaries}\label{s prelim}

We recall the Lax system for the NLS equation (\ref{nls}), its connection with the linearized NLS equation, and the Darboux transformation for the NLS equation and its Lax system.

For our purpose, it is convenient to write $\psi(x,t) = u(x,t) e^{it}$, where $u$ satisfies the normalized NLS equation 
\begin{equation}
i u_t + \frac{1}{2}u_{xx} + (|u|^2 - 1) u = 0.
\label{nls-u}
\end{equation}
The constant-amplitude wave $\psi(x,t) = e^{it}$ of the NLS equation \eqref{nls} becomes $u(x,t) =1$ and the breathers \eqref{AB}, \eqref{KM},  and \eqref{P} provide exact solutions of the normalized equaiton \eqref{nls-u} without the factor $e^{it}$ in these formulas.

\subsection{Lax system}\label{ss lin nls}

The normalized NLS equation \eqref{nls-u} for $u = u(x,t)$ is a compatibility condition $\varphi_{xt} = \varphi_{tx}$ for a $2$-vector $\varphi = \varphi(x,t)$ satisfying the Lax system
\begin{equation}
\label{lax-1}
\varphi_x = U(u,\lambda) \varphi, \quad 
U(u,\lambda) = \left(\begin{array}{cc}
\lambda & u \\
- \bar{u} & -\lambda
\end{array}
\right)
\end{equation}
and
\begin{equation}
\label{lax-2}
\varphi_t = V(u,\lambda) \varphi, \quad 
V(u,\lambda) = i \left(\begin{array}{cc}
\lambda^2 + \frac{1}{2} (|u|^2 - 1) & \lambda u + \frac{1}{2} u_x \\
-\lambda \bar{u} + \frac{1}{2} \bar{u}_x & -\lambda^2 - \frac{1}{2} (|u|^2-1)
\end{array}
\right),
\end{equation}
where $\lambda$ a complex number. The $x$-derivative equation (\ref{lax-1}) is the Zakharov--Shabat (ZS) spectral problem, which is a particular case of the AKNS spectral problem; see pioneering works \cite{ZS} and \cite{AKNS}. 
The $t$-derivative equation (\ref{lax-2}) gives the time evolution 
of the solution $\varphi(x,t)$ of the ZS spectral problem (\ref{lax-1}).

Spatially bounded solutions of the Lax system are referred to as {\em eigenfunctions} and the corresponding values $\lambda$ as {\em eigenvalues}. The set of eigenvalues $\lambda$ form the {\em Lax spectrum} of the ZS spectral problem (\ref{lax-1}). Rigorously, this terminology corresponds to considering the ZS spectral problem in the space $C_b^0(\mathbb R)$ of $x$-dependent functions which are bounded and continuous on~$\mathbb R$. However, depending on the properties of the solution $u = u(x,t)$ to the NLS equation (\ref{nls-u}) other function spaces may be considered as, for instance, the space of $L$-periodic functions $L^2_{\rm per}(0,L)$, or the space of $L$-antiperiodic functions $L^2_{\rm antiper}(0,L)$, or the space of localized functions $L^2(\mathbb{R})$. The choice of the function space affects the nature of the Lax spectrum, as this is usual for spectra of differential operators. For the spaces mentioned above, the Lax spectrum is a purely point spectrum consisting of isolated eigenvalues for $L^2_{\rm per}(0,L)$, or $L^2_{\rm antiper}(0,L)$, whereas it is a purely continuous spectrum, up to possibly a finite number of eigenvalues for $L^2(\mathbb{R})$.

The ZS spectral problem (\ref{lax-1}) can be rewritten as a classical eigenvalue problem
\begin{equation}
\label{eigen}
\left( \mathcal{L} - \lambda I \right) \varphi = 0, \quad 
\mathcal{L}:= \left( \begin{array}{cc} \partial_x & -u \\
-\bar{u} & -\partial_x \end{array} \right). 
\end{equation}
In particular, this allows to define generalized eigenfunctions and algebraic multiplicities of eigenvalues in the usual way by the bounded solutions 
of $(\mathcal{L}- \lambda I)^k \varphi = 0$ for $k \in \mathbb{N}$. 
If $\lambda$ is a double eigenvalue with the only eigenfunction $\varphi$ satisfying (\ref{eigen}), then there exists a generalized eigenfunction $\varphi_g$ satisfying the nonhomogeneous linear equation 
\begin{equation}
\label{eigen-generalized}
\left( \mathcal{L} - \lambda I \right) \varphi_g = \varphi.
\end{equation}
In this case, $\lambda$ has geometric multiplicity {\em one} and algebraic 
multiplicity {\em two}.

\begin{remark}
	\label{remark-symmetry}
Solutions of the Lax equations (\ref{lax-1}) and (\ref{lax-2}) satisfy the following symmetry. If $\varphi = (p,q)^T$ is a solution for $\lambda$, 
then $\varphi = (-\bar{q},\bar{p})^T$ is a solution for $-\bar{\lambda}$.
\end{remark}

Taking a solution $u = u(x,t)$ to the normalized NLS equation (\ref{nls-u}), solutions $v = v(x,t)$ of the corresponding linearized NLS equation 
	\begin{equation}
	i v_t + \frac{1}{2} v_{xx} + (2 |u|^2 - 1) v + u^2 \bar{v} = 0,
	\label{nls-lin}
	\end{equation}
 can be constructed from solutions $\varphi = \varphi(x,t)$ of the Lax system (\ref{lax-1})--(\ref{lax-2}). The following well-known property is a result of a straightforward calculation.

\begin{proposition}
	\label{prop-NLS-lin} Assume $u$ is a solution to the normalized NLS equation (\ref{nls-u}).
If $\varphi = (\varphi_1,\varphi_2)^T$ is a solution 
	to the Lax system (\ref{lax-1})--(\ref{lax-2}) for some $\lambda$, then 
	\begin{equation}
	\label{v-relation}
	v = \varphi_1^2 - \bar{\varphi}_2^2,  \quad \bar{v} = -\varphi_2^2 + \bar{\varphi}_1^2
	\end{equation}
	and
	\begin{equation}
	\label{v-relation-another}
	v = i(\varphi_1^2 + \bar{\varphi}_2^2), \quad \bar{v} = -i(\varphi_2^2 + \bar{\varphi}_1^2) 
	\end{equation}
are solutions to the linearized NLS equation (\ref{nls-lin}).
\end{proposition}

\begin{proof}
Due to the symmetry in Remark \ref{remark-symmetry} and the linear superposition principle, it is sufficient 
	to confirm the relations (\ref{v-relation}) and (\ref{v-relation-another}) by using $v = \varphi_1^2$ and $\bar{v} = -\varphi_2^2$. This is obtained directly:
	\begin{eqnarray*}
		&& i v_t + \frac{1}{2} v_{xx} + (2 |u|^2 - 1) v + u^2 \bar{v} \\
		&& = \varphi_1 (2 i \varphi_{1t} + \varphi_{1xx}) + (\varphi_{1x})^2 
		+ (2|u|^2-1) \varphi_1^2 - u^2 \varphi_2^2 \\
		&& = \varphi_1 ((1-|u|^2) \varphi_1 - 2 \lambda^2 \varphi_1 
		-2 \lambda u \varphi_2 - u_x \varphi_2 + \lambda (\lambda \varphi_1 + u \varphi_2) + u_x \varphi_2 + u (-\bar{u} \varphi_1 - \lambda \varphi_2)) \\
		&& \phantom{t} + (\lambda \varphi_1 + u \varphi_2)^2 + 
		(2|u|^2 -1) \varphi_1^2 - u^2 \varphi_2^2 \\
		&& = 0.
	\end{eqnarray*}
Extending the solution by using (\ref{v-relation}) and (\ref{v-relation-another}) ensures that $\bar{v}$ is a complex conjugate of $v$.
\end{proof}

\begin{remark}
	\label{remark-completeness}
Solutions $\varphi = \varphi(x,t)$ to the Lax system (\ref{lax-1})--(\ref{lax-2}) which are bounded functions in $x$ generate bounded solutions $v = v(x,t)$ to the linearized NLS equation (\ref{nls-lin}) by means of the transformations (\ref{v-relation}) and (\ref{v-relation-another}). On the other hand, solutions $\varphi = \varphi(x,t)$ which are unbounded functions in $x$ generate unbounded solutions $v = v(x,t)$ but the linear superposition of unbounded solutions may become bounded \cite{GS-2021}. This latter property must be taken into account when constructing solutions to the linearized NLS equation (\ref{nls-lin}) either in $L^2_{\rm per}(0,L)$ or in $L^2(\mathbb{R})$ by using Proposition \ref{prop-NLS-lin}.
\end{remark}

The result in Proposition~\ref{prop-NLS-lin} can be extended by taking two linearly independent solutions $\varphi = (\varphi_1,\varphi_2)^T$ and $\phi = (\phi_1,\phi_2)^T$ to the Lax system (\ref{lax-1})--(\ref{lax-2}) for the same value of $\lambda$. Then from these two solutions we can construct the three pairs of solutions of the linearized NLS equation (\ref{nls-lin}) given in Table \ref{table-1}. The symmetry of the Lax system in Remark~\ref{remark-symmetry} implies that the solutions of the Lax system for $-\bar\lambda$ lead, up to sign, to the same solutions of the linearized NLS equation~(\ref{nls-lin}).

\begin{table}[hbtp!]
	\begin{center}
		\begin{tabular}{|c|c|c|}
			\hline
		Pair I  & Pair II & Pair III 	\\				 
			\hline
			& & \\
$v = \varphi_1^2 - \bar{\varphi}_2^2$ & $v = \varphi_1 \phi_1 - \bar{\varphi}_2 \bar{\phi}_2$ & $v = \phi_1^2 - \bar{\phi}_2^2$	\\
			\hline
				& & \\
			$v = i \varphi_1^2 + i \bar{\varphi}_2^2$ & $v = i \varphi_1 \phi_1  + i \bar{\varphi}_2 \bar{\phi}_2$ & $v = i\phi_1^2 + i \bar{\phi}_2^2$	\\
			\hline
		\end{tabular}
        \end{center}
\vspace*{1ex}        
        \caption{Table of possible solutions of the linearized NLS equation (\ref{nls-lin}) generated from two solutions $\varphi = (\varphi_1,\varphi_2)^T$ and $\phi = (\phi_1,\phi_2)^T$  to the Lax system (\ref{lax-1})--(\ref{lax-2}) for the same value of $\lambda$.}
	\label{table-1}
\end{table}

\begin{remark}
	\label{remark-double}
	If $\lambda$ is a double eigenvalue with the only eigenfunction 
	$\varphi = (\varphi_1,\varphi_2)^T$ satisfying (\ref{eigen}) 
	and the generalized eigenfunction 
	$\varphi_g = (\varphi_{g1},\varphi_{g2})^T$ satisfying  (\ref{eigen-generalized}), then the linearized NLS equation 
	(\ref{nls-lin}) admits the solutions 
\begin{equation}
\label{solutions-generalized}
	v = 2 \varphi_1 \varphi_{g1} - 2 \bar{\varphi}_2 \bar{\varphi}_{g2}, \quad
	v = 2 i \varphi_1 \varphi_{g1} + 2 i \bar{\varphi}_2 \bar{\varphi}_{g2},
\end{equation}
	 in addition to the two solutions in Pair I of Table \ref{table-1}.
\end{remark}

\subsection{Darboux transformation}

For the construction of breathers, we use the following version of the one-fold Darboux transformation from \cite[Propositions 2.2 and 3.1]{ContPel}. 

\begin{proposition}
  \label{prop-Darboux}
Assume that $u = u_0(x,t)$ is a solution to the normalized NLS equation (\ref{nls-u}) and pick $\lambda_0 \in \mathbb{C}$. If $\varphi = (p_0,q_0)^T$ is a particular solution of the Lax system (\ref{lax-1})--(\ref{lax-2}) with $u = u_0$ and $\lambda = \lambda_0$, then
\begin{equation}
\label{DT-potential}
\hat{u}_0 = u_0 + \frac{2(\lambda_0 + \bar{\lambda}_0) p_0 \bar{q}_0}{|p_0|^2 + |q_0|^2}
\end{equation}
is a solution to the normalized NLS equation (\ref{nls-u}) and
  $\varphi = (\hat{p}_0,\hat{q}_0)^T$ with
\begin{equation}
\label{DT-eigen}
\left[ \begin{array}{l} \hat{p}_0 \\ \hat{q}_0 \end{array} \right] = \frac{\lambda_0 + \bar{\lambda}_0}{|p_0|^2 + |q_0|^2} \left[ \begin{array}{l} -\bar{q}_0 \\ \bar{p}_0\end{array} \right]
\end{equation}
 is a particular solution of the Lax system (\ref{lax-1})--(\ref{lax-2}) with $u = \hat{u}_0$ and $\lambda = \lambda_0$. Furthermore, the following identity holds:
\begin{equation}
\label{DT-squared}
|\hat{u}_0|^2 = |u_0|^2 + \frac{\partial^2}{\partial x^2} \log(|p_0|^2 + |q_0|^2).
\end{equation}
\end{proposition}

\begin{remark}
	By the symmetry in Remark \ref{remark-symmetry}, $\varphi = (-\bar{q}_0,\bar{p}_0)^T$ is a solution of the Lax system (\ref{lax-1})--(\ref{lax-2}) with $u = u_0$ and $\lambda = -\bar{\lambda}_0$, 
	whereas $\varphi = (-\bar{\hat{q}}_0,\bar{\hat{p}}_0)^T$ is a solution of the Lax system for $u = \hat{u}_0$ and $\lambda = -\bar{\lambda}_0$.
\end{remark}

\begin{remark}
The result in Proposition~\ref{prop-Darboux} provides new solutions to the normalized NLS equation (\ref{nls-u}), and to the associated Lax system (\ref{lax-1})--(\ref{lax-2}), when $\lambda_0+\bar\lambda_0\not=0$, i.e., when $\lambda_0$ is not purely imaginary. When 
$\lambda_0+\bar\lambda_0=0$, it gives the same solution $\hat u_0=u_0$ to the normalized NLS equation (\ref{nls-u}) and the trivial solution $\varphi = (0,0)^T$ to the Lax system (\ref{lax-1})--(\ref{lax-2}). Breathers are found by taking $u_0=1$ and positive values $\lambda_0$: $\lambda_0\in(0,1)$ for AB, $\lambda_0\in(1,\infty)$ for KMB, and $\lambda_0=1$ for PRW. 
\end{remark}

In addition to the Darboux transformation $u_0 \mapsto \hat{u}_0$ in Proposition~\ref{prop-Darboux}, we have a Darboux transformation 
$\Phi(\lambda) \mapsto \hat{\Phi}(\lambda)$ between solutions of the Lax system 
(\ref{lax-1})--(\ref{lax-2}). More precisely, assuming that $\Phi(\lambda)$ is a $2\times2$ matrix solution to the Lax system with $u = u_0$, then 
\begin{equation}
\label{fund-matrix}
\hat{\Phi}(\lambda) = D(\lambda) \Phi(\lambda) 
\end{equation}
is a $2\times2$ matrix solution to the Lax system with $u = \hat{u}_0$ 
if $\lambda \neq \{ \lambda_0,-\bar{\lambda}_0\}$, 
where the Darboux matrix $D(\lambda)$ is given by
\begin{equation}
\label{DT-matrix}
D(\lambda) := I + \frac{1}{\lambda - \lambda_0} \left[ \begin{array}{l} \hat{p}_0 \\ \hat{q}_0 \end{array} \right] \left[ -q_0 \;\; p_0 \right] 
\end{equation}
and $I$ stands for the $2\times2$ identity matrix. Since 
$$
\det D(\lambda) = \frac{\lambda+\bar\lambda_0}{\lambda-\lambda_0}, 
$$
the matrix $D(\lambda)$ is invertible, and the correspondence between the $2\times2$ matrix solutions  $\Phi(\lambda)$ and  $\hat\Phi(\lambda)$ is one-to-one, when $\lambda \neq \{ \lambda_0,-\bar{\lambda}_0\}$.

\section{Constant-amplitude background}\label{s constant}

Here we discuss the simple case of the constant solution $u=1$. We determine the Lax spectrum and compare the set of solutions of the linearized NLS equation
	\begin{equation}
i v_t + \frac{1}{2} v_{xx} + v + \bar{v} = 0
\label{nls-lin-constant}
        \end{equation}
obtained using standard tools of Fourier analysis with the one given by the solutions of the Lax system \eqref{lax-1}--\eqref{lax-2}. This comparison will be useful in the study of linear instability of AB and KMB in Sections \ref{s AB} and \ref{s KMB} respectively.

\subsection{Lax spectrum}

Since the solution $u=1$ is constant, the Lax system (\ref{lax-1})--(\ref{lax-2}) can be solved explicitly. 
Two linearly independent solutions exist for every 
$\lambda$ since the Lax system (\ref{lax-1})--(\ref{lax-2}) is of the second order. We only consider real and purely imaginary values of $\lambda$ because the solutions found for the other complex values $\lambda$ are unbounded.
  
For $\lambda \in \mathbb{R}_+$, two solutions to the Lax equations (\ref{lax-1})--(\ref{lax-2}) are given by:
\begin{equation}
\label{eigenvectors-constant-background}
\varphi = \left[ \begin{array}{l} \sqrt{\lambda - \frac{i}{2} k(\lambda)} \\ 
- \sqrt{\lambda + \frac{i}{2} k(\lambda)} \end{array} \right] e^{-\frac{1}{2} i k(\lambda) (x + i\lambda t)}, \quad 
\phi = \left[ \begin{array}{l} \sqrt{\lambda + \frac{i}{2} k(\lambda)} \\ 
- \sqrt{\lambda - \frac{i}{2} k(\lambda)} \end{array} \right] e^{+\frac{1}{2} ik(\lambda) (x + i\lambda t)}, 
\end{equation}
where $k(\lambda) := 2 \sqrt{1-\lambda^2}$.
These solutions are bounded for $\lambda \in (0,1]$ and are linearly independent for $\lambda \neq 1$, that is, for  $k(\lambda)\neq 0$. For $\lambda=1$, two linearly independent solutions are given by
\begin{equation}
\label{eigenvectors-constant-background-zero}
\lambda = 1 : \quad 
\varphi = \left[ \begin{array}{l} 1\\ 
-1 \end{array} \right], \quad 
\phi = \left[  \begin{array}{l} x+it + 1 \\ -x-it \end{array} \right].
\end{equation}
Solutions for $\lambda\in\mathbb R_-$, and in particular, for $\lambda \in [-1,0)$, are found from the symmetry property of the Lax equations in Remark~\ref{remark-symmetry}. This implies that any $\lambda\in(-1,0)\cup(0,1)$ is a geometrically double eigenvalue, whereas $\lambda=\pm1$ are geometrically simple. 

For $\lambda = i \gamma$ with $\gamma \in \mathbb{R}$, two solutions to the Lax equations (\ref{lax-1})--(\ref{lax-2}) are given by:
\begin{equation}
\label{eigenvectors-constant-background-gamma}
\varphi = \left[ \begin{array}{l} \sqrt{\frac{1}{2} k(\gamma) - \gamma} \\ 
- i \sqrt{\frac{1}{2} k(\gamma) + \gamma} \end{array} \right] e^{-\frac{1}{2} i k(\gamma) (x- \gamma t)}, \quad 
\phi = \left[ \begin{array}{l} \sqrt{\frac{1}{2} k(\gamma) + \gamma} \\ 
i \sqrt{\frac{1}{2} k(\gamma) - \gamma} \end{array} \right] e^{+\frac{1}{2} ik(\gamma) (x- \gamma t)}, 
\end{equation}
where $k(\gamma) := 2 \sqrt{1 + \gamma^2}$.
These solutions are bounded and linearly independent for every $\gamma \in \mathbb R_+$. Solutions for $\gamma \in \mathbb R_-$ are found from the symmetry property of the Lax equations in Remark~\ref{remark-symmetry}. Consequently, any $\lambda = i \gamma$ with $\gamma \in \mathbb{R} \backslash \{0\}$ is a geometrically double eigenvalue.

For $\lambda = 0$ ($\gamma = 0$), there are two linearly independent solutions,
\begin{equation}
\label{eigenvectors-constant-background-two}
\lambda = 0 : \quad \varphi = \left[ \begin{array}{c} 1\\ 
-i \end{array} \right] e^{-i x}, \quad 
\phi = \left[  \begin{array}{c} 1 \\ i \end{array} \right] e^{+ix},
\end{equation}
implying that $\lambda=0$ is a geometrically double eigenvalue. In contrast to the eigenvalues above, the eigenvalue $\lambda = 0$ has algebraic multiplicity four because the bounded solutions of $\mathcal{L}^2 \varphi = 0$ are spanned by (\ref{eigenvectors-constant-background-two}) and two additional solutions 
\begin{equation}
\label{eigenvectors-constant-background-two-generalized}
\lambda = 0 : \quad \varphi_{\rm g} = \left[ \begin{array}{c} t\\ 
-1 - i t \end{array} \right] e^{-i x}, \quad 
\phi_{\rm g} = \left[ \begin{array}{c} -t \\ 
-1 - it \end{array} \right] e^{+i x}.
\end{equation}

These computations are summarized in the following lemma, where 
we have also checked algebraic multiplicities of all eigenvalues.

\begin{lemma}
	\label{lemma-spectrum}
The Lax spectrum of the spectral problem (\ref{lax-1}) with $u = 1$ in the space $C_b^0(\mathbb{R})$ of bounded continuous functions is the set
\begin{equation}
	\label{spectrum-constant}
	\Sigma_0 =  i\mathbb{R} \cup [-1,1],
\end{equation}
and the following properties hold:
\begin{enumerate}
\item $\lambda = \pm 1$  are algebraically simple eigenvalues;
\item each $\lambda \in \Sigma_0 \backslash \{0,\pm 1\}$ is a geometrically and algebraically  double eigenvalue;
\item $\lambda = 0$ is an eigenvalue with geometric multiplicity two and algebraic multiplicity four.
\end{enumerate}
\end{lemma}

\begin{proof}
	Geometric multiplicity of all eigenvalues has been checked with direct computations resulting in (\ref{eigenvectors-constant-background}), (\ref{eigenvectors-constant-background-zero}), (\ref{eigenvectors-constant-background-gamma}), and (\ref{eigenvectors-constant-background-two}). 
In order to check the algebraic multiplicity of eigenvalues, 
we note that for each eigenvalue $\lambda$, 
the bounded eigenfunctions $\varphi$ and $\phi$ in $C^0_b(\mathbb{R})$ are periodic in $x$ with some spatial period $L$. For 
the algebraic multiplicity of $\lambda$, we need to solve $(\mathcal{L} - \lambda I) \varphi_g = \varphi$ and $(\mathcal{L} - \lambda I) \phi_g = \phi$ in the space of periodic functions with the same period $L$. Consequently, we can check the Fredholm condition in $L^2(0,L)$ equipped with the standard inner product $\langle \cdot, \cdot \rangle$.

Let $\varphi = (\varphi_1,\varphi_2)$ be the bounded eigenfunction of the eigenvalue problem $(\mathcal{L}- \lambda I) \varphi = 0$. By the symmetry, the adjoint problem $\left( \mathcal{L}^* - \bar{\lambda} I \right) \varphi^* = 0$ admits the eigenfunction $\varphi^* = (\bar{\varphi}_2,\bar{\varphi}_1)^T$. If $\lambda \in \Sigma_0 \backslash \{+1,-1\}$, there exists another linearly independent eigenfunction $\phi = (\phi_1,\phi_2)^T$, for which we have similarly $\phi^* = (\bar{\phi}_2,\bar{\phi}_1)^T$. Since 
$\langle \psi^*,\varphi \rangle = \langle \varphi^*, \phi \rangle = 0$, the generalized eigenfunctions $\varphi_g$ and $\phi_g$ 
exist if and only if $\langle \varphi^*, \varphi \rangle = 0$ and $\langle \phi^*, \phi \rangle = 0$.

For $\lambda \in (0,1)$, we obtain 
	\begin{eqnarray*}
\langle \varphi^*, \varphi \rangle = -2 \lambda L e^{\lambda k(\lambda) t}, \quad 
\langle \phi^*, \phi \rangle = -2 \lambda L e^{-\lambda k(\lambda) t},
	\end{eqnarray*}
which are both nonzero for $\lambda \neq 0$. For $\lambda = 1$, 
only one linearly independent eigenfunction $\varphi$ in  (\ref{eigenvectors-constant-background-zero}) exists and we 
check that $\langle \varphi^*,\varphi \rangle = -2 L \neq 0$.
For $\lambda = i \gamma$ with $\gamma \in \mathbb{R}$, we obtain 
	\begin{eqnarray*}
	\langle \varphi^*, \varphi \rangle = -2 i \gamma L e^{i \gamma k(\gamma) t}, \quad 
	\langle \phi^*, \phi \rangle = -2 i \gamma L e^{-i \gamma k(\gamma) t},
\end{eqnarray*}
which are both nonzero for $\gamma \neq 0$.
Hence, the algebraic multiplicity of all nonzero eigenvalues is equal to their geometric multiplicity.

For the eigenvalue $\lambda=0$ with the eigenfunctions (\ref{eigenvectors-constant-background-two}), we obtain $\langle \varphi^*, \varphi \rangle = \langle \phi^*, \phi \rangle = 0$, in agreement with the existence of the generalized eigenfunctions (\ref{eigenvectors-constant-background-two-generalized}). On the other hand, 
we also have 
$$
\langle \varphi^*, \varphi_g \rangle = -L, \quad 
\langle \phi^*, \phi \rangle = -L,
$$
which implies that no new generalized eigenfunctions satisfying $\mathcal{L}^3 \varphi = 0$ exist. Hence, the zero eigenvalue has algebraic multiplicity equal to four.	
\end{proof}


Replacing the space $C_b^0(\mathbb{R})$ by $L^2(\mathbb R)$ in Lemma~\ref{lemma-spectrum} the Lax spectrum does not change, the difference being  that $\Sigma_0$ becomes a purely continuous spectrum in $L^2(\mathbb R)$. In the space $L^2_{\rm per}(0,L)$ of $L$-periodic functions, the Lax spectrum only contains the eigenvalues $\lambda\in\Sigma_0$ with $L$-periodic associated eigenfunctions, hence the purely point spectrum is located at 
\begin{equation}
\label{spectrum-constant-periodic}
 \Sigma_0^{(P)}= \{ \pm \lambda^{(P)}_m, \;\; m \in \{0,\mathbb{N}_{\rm even}\}\}, \quad
\lambda^{(P)}_m := \sqrt{1-\frac{\pi^2}{L^2} m^2}.
\end{equation}   
Similarly, in the space $L^2_{\rm antiper}(0,L)$ of $L$-antiperiodic functions, the Lax spectrum only contains the eigenvalues $\lambda\in\Sigma_0$ with $L$-antiperiodic associated eigenfunctions, hence the purely point spectrum is located at 
\begin{equation}
\label{spectrum-constant-antiperiodic}
 \Sigma_0^{(A)}= \{ \pm \lambda^{(A)}_m, \;\; m \in \mathbb{N}_{\rm odd}\}, \quad
\lambda^{(A)}_m := \sqrt{1-\frac{\pi^2}{L^2} m^2}.
\end{equation}
The algebraic and geometric multiplicities of these eigenvalues remain the same, as given by Lemma~\ref{lemma-spectrum}. Notice that $\lambda=0$ is an eigenvalue only for particular periods $L\in\pi\mathbb N$.

Figure \ref{fig-Lax} illustrates these results. The left panel shows the purely continuous spectrum of $\Sigma_0$ in $L^2(\mathbb{R})$ given by (\ref{spectrum-constant}). The right panel shows the union $\Sigma_0^{(P)} \cup \Sigma_0^{(A)}$ of the purely point spectra in $L^2_{\rm per}(0,L)$ and $L^2_{\rm antiper}(0,L)$ given by (\ref{spectrum-constant-periodic}) and (\ref{spectrum-constant-antiperiodic}), respectively.

\begin{figure}[htb]
	\begin{center}
		\includegraphics[width=6cm,height=4cm]{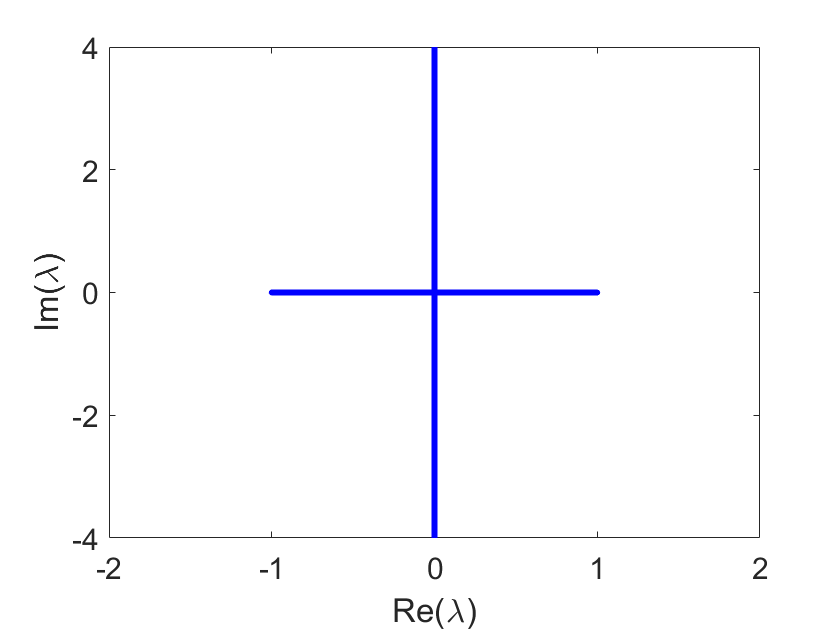}
		\includegraphics[width=6cm,height=4cm]{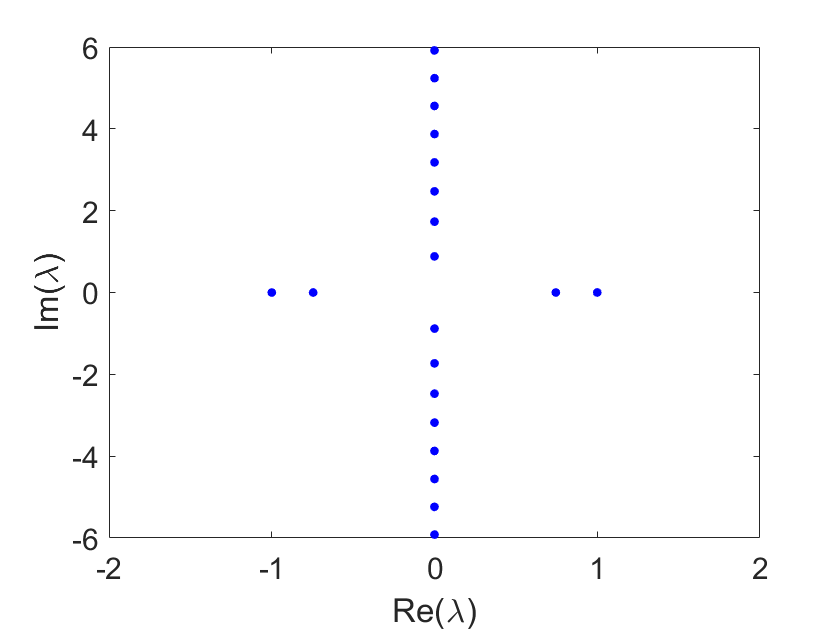}
	\end{center}
	\caption{Left: The Lax spectrum $\Sigma_0$ in $L^2(\mathbb{R})$. 
		Right: The union  $\Sigma_0^{(P)} \cup \Sigma_0^{(A)}$ of the Lax spectra in $L^2_{\rm per}(0,L)$ and $L^2_{\rm antiper}(0,L)$ for some positive $L\notin\pi\mathbb N$.}
	\label{fig-Lax}
\end{figure}

\subsection{Localized solutions}

Since the linearized NLS equation \eqref{nls-lin-constant} has constant coefficients, the Fourier transform provides a basis of bounded solutions in $x$ which can be used to represent a general solution in $L^2(\mathbb{R})$. 
The following proposition gives the result.

\begin{proposition}
\label{theorem-localized solutions}
For every $v_0 \in L^2(\mathbb{R})$, there exists a unique solution 
$v \in C^0(\mathbb{R},L^2(\mathbb{R}))$ to the linearized NLS equation (\ref{nls-lin-constant}) satisfying $v(x,0) = v_0(x)$ in the form of a linear superposition
\begin{equation}\label{basis Fourier}
	v(x,t)=\int_0^\infty \left[c_k^+ \widetilde v_k^+(t) + c_k^- \widetilde v_k^-(t)\right] \cos(kx)dk + 
	\int_0^\infty\left[ d_k^+ \widetilde v_k^+(t) + d_k^- \widetilde v_k^-(t)\right]\sin(kx)dk,
\end{equation}
where coefficients $c_k^\pm$ and $d_k^\pm$ are uniquely found from $v_0 \in L^2(\mathbb{R})$, and the functions ${\widetilde v}_k^{\pm}(t)$ for $k \geq 0$ are given as follows:
\begin{eqnarray}
\label{vk0}
k = 0 : & \quad & \left\{ \begin{array}{l} \widetilde v_0^+(t)= 2i,\\ \widetilde v_0^-(t)= 1+2it,\end{array} \right. \\
\label{vkpm}
k \in (0,2) : & \quad & \left\{ \begin{array}{l} \widetilde v_k^+(t)=(2i\lambda+k)e^{k\lambda t},\\
\widetilde v_k^-(t)=(2i\lambda-k)e^{-k\lambda t},\end{array} \right. 
\quad \lambda=\lambda(k)=\frac{1}{2} \sqrt{4 - k^2}, \\
\label{vk2}
k = 2 : & \quad & \left\{ \begin{array}{l} \widetilde v_2^+(t)= 2,\\ 
\widetilde v_2^-(t)= i+2t, \end{array} \right. \\
\label{vkpm2}
k \in (2,\infty) : & \quad & \left\{ \begin{array}{l} \widetilde v_k^+(t)=k\cos(k\gamma t)-2i\gamma\sin(k\gamma t), \\
\widetilde v_k^-(t)=2i\gamma\cos(k\gamma t)+k\sin(k\gamma t), \end{array} \right.
\quad \gamma=\gamma(k)=\frac{1}{2} \sqrt{k^2 - 4}.
\end{eqnarray}
\end{proposition}

\begin{proof}
	The proof is based on separation of variables and straightforward computations. Indeed, substituting $v(t,x) = \widetilde v_k(t) e^{ikx}$ into (\ref{nls-lin-constant}) yields the linear differential equation 
	\[
	i \frac d{dt} \widetilde v_k +\left(1-\frac{k^2}2\right) \widetilde v_k +\overline{\widetilde v}_k = 0,
	\]
	with two linearly independent solutions $\widetilde v_k^+(t)$ and $\widetilde v_k^-(t)$ 
	given by (\ref{vk0}), (\ref{vkpm}), (\ref{vk2}), and (\ref{vkpm2}) for different values of $k \geq 0$. Completeness of the basis of bounded 
	functions in $L^2(\mathbb{R})$ is given by the Fourier theory.
\end{proof}

\begin{remark}\label{spectrum bounded}
Since the points $k = 0$ and $k = 2$ are of measure zero in the integral (\ref{basis Fourier}) we actually do not need the solutions for $k=0$ and $k=2$. However these solutions play a role when the space $L^2(\mathbb R)$ is replaced by the space $L^2_{\rm per}(0,L)$ of $L$-periodic functions.
\end{remark}

\begin{remark}
From the Fourier decomposition \eqref{basis Fourier} we can determine the spectrum of the linear operator from the linearized NLS equation (\ref{nls-lin-constant}), when acting in the space $L^2(\mathbb{R})$. We find 
that the purely continuous spectrum is located at 
$\{\pm k\lambda(k) \;;\; k\geq0\} = i\mathbb R \cup[-1,1]$. This implies the spectral instability of $u=1$ in the linearized NLS equation (\ref{nls-lin-constant}).
\end{remark}

It follows from Proposition \ref{prop-NLS-lin} that solutions of the linearized NLS equation \eqref{nls-lin-constant} can be constructed from solutions of the Lax equations (\ref{lax-1})--(\ref{lax-2}) with $u = 1$ and $\lambda\in\mathbb C$.
We show below how to recover the Fourier basis in the decomposition \eqref{basis Fourier} from the eigenfunctions associated to the Lax spectrum in Lemma~\ref{lemma-spectrum}. We use the three pairs of solutions given in Table~\ref{table-1}.

\vspace*{1.5ex}

{\bf Pair~II of Table~\ref{table-1}.} Using $\varphi$ and $\phi$ in either \eqref{eigenvectors-constant-background} or \eqref{eigenvectors-constant-background-gamma},  
we obtain the same constant solutions
\begin{equation}\label{v0+}
v=0,\quad \widetilde v_0^+(t)=2i,
\end{equation}
where $\widetilde v_0^+$ is the same as in (\ref{vk0}). Using $\varphi$ and $\phi$ from \eqref{eigenvectors-constant-background-zero}, we find the solutions
\begin{equation}\label{v0-}
{\widetilde v}_0^-(t) = 1 + 2it,\quad v(x) = i(2x+1),
\end{equation}
where $\widetilde v_0^-$ is the same as in (\ref{vk0}). The two bounded solutions in the Fourier decomposition \eqref{basis Fourier} with $k=0$ are recovered.

\vspace*{1.5ex}

{\bf Pairs~I and III of Table~\ref{table-1}.}
Using the eigenfunction $\varphi$ from \eqref{eigenvectors-constant-background-zero} associated to the simple eigenvalue $\lambda=1$ we obtain the solutions from \eqref{v0+}, again. By the symmetry of the Lax system in Remark~\ref{remark-symmetry}, the solutions obtained for $\lambda=-1$ are, up to sign, the same.

Next, using $\varphi$ and $\phi$ in \eqref{eigenvectors-constant-background}  for $\lambda\in(0,1)$,  we find the following four linearly independent bounded solutions:
\begin{equation}
  \label{Fourier-basis-I}
  v_\lambda^{+}(x,t) = -(2i \lambda + k) e^{\lambda k t}  \sin(kx) ,
  \quad
v_{-\lambda}^{+}(x,t) = (2i \lambda + k) e^{\lambda k t} \cos(k x),
\end{equation}
and
\begin{equation}
\label{Fourier-basis-III}
v_\lambda^{-}(x,t) = (2i \lambda - k) e^{-\lambda k t} \sin(k x), \quad
v_{-\lambda}^{-}(x,t) = (2i \lambda - k) e^{-\lambda k t} \cos(k x),
\end{equation}
in which $k=k(\lambda)\in(0,2)$. These are, up to sign, equal to the four solutions in the Fourier decomposition \eqref{basis Fourier} given by \eqref{vkpm} so that we have a one-to-one correspondence between the solutions provided by the Lax system with $\lambda\in(0,1)$ and the solutions in \eqref{basis Fourier} with $k\in(0,2)$ through the equalities $k=k(\lambda)$ and $\lambda=\lambda(k)$. By the symmetry of the Lax system in Remark~\ref{remark-symmetry}, the solutions obtained for $\lambda=(-1,0)$ are, up to sign, the same.

Using $\varphi$ and $\phi$ in \eqref{eigenvectors-constant-background-gamma}  for $\lambda=i\gamma$ with $\gamma\in\mathbb R_+$, we only find two linearly independent solutions
\begin{equation}
\label{Fourier-basis-1-new}
\begin{array}{l}
v_{\lambda}^+(x,t) =  k\cos( k x- k\gamma t) + 2 i \gamma \sin( kx - k\gamma t),\\
v_{-\lambda}^+(x,t) =  -2 i \gamma \cos(kx -k \gamma t) + k\sin( kx-k\gamma t),
\end{array}
\end{equation}
in which $k=k(\gamma) \in (2,\infty)$. However, using $\varphi$ and $\phi$ in \eqref{eigenvectors-constant-background-gamma} with $-\gamma$ instead of $\gamma$, 
we obtain other two linearly independent solutions, 
\begin{equation}
\label{Fourier-basis-2-new}
\begin{array}{l}
v_{\lambda}^-(x,t) = k\cos (kx+ k\gamma t) - 2 i \gamma \sin (kx + k\gamma t), \\
v_{-\lambda}^-(x,t) = 2 i \gamma \cos (kx +k \gamma t) + k \sin (kx +k \gamma t).
\end{array}
\end{equation}
Solutions (\ref{Fourier-basis-1-new}) and (\ref{Fourier-basis-2-new}) are linear combinations of the four solutions in the Fourier decomposition \eqref{basis Fourier} with $k\in(2,\infty)$ given by \eqref{vkpm2}, and we have a one-to-one correspondence between these solutions through the equalities $k = k(\gamma)$ and $\gamma = \gamma(k)$.

Finally, using $\varphi$ and $\phi$ in \eqref{eigenvectors-constant-background-two}  for $\lambda=0$, we obtain two linearly independent solutions
\begin{equation}\label{Fourier 0}
v_0^+(x,t) = -2 \sin(2x),\quad 
v_{-0}^+(x,t) = 2 \cos(2x).
\end{equation}
These recover the two solutions with $k=2$ in the Fourier decomposition \eqref{basis Fourier} corresponding to $\widetilde v_2^+$ in \eqref{vk2}. 
In order to recover the two solutions given by $\widetilde v_2^-$ in (\ref{vk2}), we use (\ref{solutions-generalized}) with the eigenfunctions (\ref{eigenvectors-constant-background-two}) and the generalized eigenfunctions (\ref{eigenvectors-constant-background-two-generalized}) to obtain
\begin{equation}\label{Fourier 0 gen}
v_0^-(x,t) = 2 (i + 2t) \cos(2x) - 2 \sin(2x),\quad 
v_{-0}^-(x,t) = 2 (i + 2t) \sin(2x) + 2 \cos(2x).
\end{equation}
Using (\ref{solutions-generalized}) with $\phi$ and $\phi_g$ produces 
the same solutions as (\ref{Fourier 0 gen}) up to the change of signs. 
Solutions (\ref{Fourier 0}) and (\ref{Fourier 0 gen}) for $\lambda = 0$ recover the four solutions in the Fourier decomposition \eqref{basis Fourier} given by (\ref{vk2}) for $k=2$.

Summarizing, the set of eigenfunctions of the Lax equations (\ref{lax-1})--(\ref{lax-2}) with $u = 1$ and $\lambda \in \Sigma_0$ allows us to recover the Fourier basis in the decomposition \eqref{basis Fourier}, except for the two functions $\widetilde v_2^-(t)\cos(2x)$ and $\widetilde v_2^-(t)\sin(2x)$ with $k=2$. The entire basis is recovered when also using the generalized eigenfunctions (\ref{eigenvectors-constant-background-two-generalized})
associated to the eigenvalue $\lambda=0$. This leads to an alternative expansion for solutions $v \in C^0(\mathbb{R},L^2(\mathbb{R}))$ to the linearized NLS equation (\ref{nls-lin-constant}),
\begin{equation}\label{basis Fourier 2}
	v(x,t)=\int_0^\infty \left[c_k^+  v_{\lambda(k)}^+(x,t) + c_k^-v_{\lambda(k)}^-(x,t)  + c_{-k}^+  v_{-\lambda(k)}^+(x,t) + c_{-k}^-v_{-\lambda(k)}^-(x,t)\right] dk,
\end{equation}
where coefficients $c_{\pm k}^\pm$ are uniquely defined from the initial condition $v(\cdot,0)=v_0 \in L^2(\mathbb R)$, and $v_{\pm\lambda(k)}^\pm(x,t)$ are given by  \eqref{Fourier-basis-I}--\eqref{Fourier-basis-III} if $k\in(0, 2)$ and by \eqref{Fourier-basis-1-new}--\eqref{Fourier-basis-2-new} if $k\in(2,\infty)$. Since the points $k = 0$ and $k = 2$ are of measure zero in the integral (\ref{basis Fourier 2}) we do not need 
solutions (\ref{v0+}), (\ref{v0-}), (\ref{Fourier 0}), and (\ref{Fourier 0 gen}).

\begin{remark}\label{localized solutions}
Since the solutions for $k = 0$ and $k = 2$ are not used in the expansion (\ref{basis Fourier 2}), the solutions found from Pair~II of Table~\ref{table-1} and from the eigenvalues $\lambda=0$ and $\lambda=\pm1$ play no role in the dynamics of localized perturbations on the background of $u = 1$. In particular, linearly growing in $t$ solutions play no role in this dynamics.
All relevant solutions are obtained using the eigenfunctions of the Lax system for $\lambda\in\Sigma_0\setminus\{0,\pm1\}$ in Pairs I and III of Table \ref{table-1}. 
\end{remark}

\subsection{Periodic solutions}\label{ss periodic1}

Solutions of the linearized NLS equation \eqref{nls-lin-constant} in the space $L^2_{\rm per}(0,L)$ of periodic functions with the fundamental period $L>0$ are found by restricting the continuous Fourier decomposition (\ref{basis Fourier}) to the discrete values 
\begin{equation}
\label{k-m}
k_m := \frac{2\pi m}{L}, \quad 
m \in \mathbb{N}_0 := \{ 0, \mathbb{N}\}.
\end{equation}
This leads to a decomposition in Fourier series
\begin{equation}\label{basis Fourier per}
	v(x,t)=\sum_{m\in\mathbb N_0} \left[c_{k_m}^+ \widetilde v_{k_m}^+(t) + c_{k_m}^- \widetilde v_{k_m}^-(t)\right] \cos({k_m}x) + 
	\sum_{m\in\mathbb N}\left[ d_{k_m}^+ \widetilde v_{k_m}^+(t) + d_{k_m}^- \widetilde v_{k_m}^-(t)\right]\sin({k_m}x),
\end{equation}
where coefficients $c_{k_m}^\pm$ and $d_{k_m}^\pm$ are uniquely found from the initial condition $v(\cdot,0)=v_0 \in L^2_{\rm per}(0,L)$, and the functions ${\widetilde v}_{k_m}^{\pm}(t)$ are given by \eqref{vk0}--\eqref{vkpm2}.

We obtain an equivalent decomposition using the eigenfunctions of the Lax system. For the Lax system we have to consider both $L$-periodic and $L$-antiperiodic solutions, because the solutions of the linearized NLS equation \eqref{nls-lin-constant} are constructed using squares of solutions of the Lax system. 

The Lax spectra $\Sigma_0^{(P)}$ in $L^2_{\rm per}(0,L)$ and $ \Sigma_0^{(A)}$ in $L^2_{\rm antiper}(0,L)$ are given in \eqref{spectrum-constant-periodic} and \eqref{spectrum-constant-antiperiodic}, respectively.
 For notational simplicity we set $\lambda(k_m)=\lambda^{(P)}_m$, if $m$ is even, and $\lambda(k_m)=\lambda^{(A)}_m$, if $m$ is odd, so that $\Sigma_0^{(P)} \cup \Sigma_0^{(A)} = \{\pm\lambda(k_m), \; m\in\mathbb N_0\}$.
  The arguments above show that all functions in the Fourier series \eqref{basis Fourier per} are recovered from the eigenfunctions of the Lax system associated to the eigenvalues $\lambda\in\Sigma_0^{(P)} \cup \Sigma_0^{(A)}$. Indeed,
for $m=0$ we have the eigenvalues $\pm\lambda(0)=\pm1\in\Sigma_0^{(P)}$ leading to the solutions $\widetilde v_0^+$ and $\widetilde v_0^-$ given by \eqref{v0+} and \eqref{v0-}, respectively, which are constant in $x$. If $0<\pi m < L$, then $\lambda(k_m)\in(0,1)$, and we have the four linearly independent solutions in \eqref{Fourier-basis-I}--\eqref{Fourier-basis-III} with $\lambda= \lambda(k_m)$. If $\pi m > L$, then $\lambda(k_m) = i \gamma(k_m)$ is purely imaginary, and we have the four linearly independent solutions in \eqref{Fourier-basis-1-new}--\eqref{Fourier-basis-2-new} with $\gamma= \gamma(k_m)$. In the particular case $L=\pi m$, for some $m\in\mathbb N$, we have $\lambda(k_m)=0$ and  four linearly independent solutions are given in \eqref{Fourier 0} and \eqref{Fourier 0 gen}.

As a consequence, an arbitrary solution of the linearized NLS equation (\ref{nls-lin-constant}) in $L^2_{\rm per}(0,L)$ can be written in the series form:
\begin{eqnarray}
  \nonumber
  v(x,t)&=& c_0^+ \widetilde v_0^+(t) + c_0^- \widetilde v_0^-(t) \\
&&  + \sum_{m \in \mathbb{N}}\left[ c_m^+ v_{\lambda(k_m)}^+(x,t) + c_m^- v_{\lambda(k_m)}^-(x,t) + c_{-m}^+ v_{-\lambda(k_m)}^+(x,t) + c_{-m}^- v_{-\lambda(k_m)}^-(x,t)\right]\!, \qquad\
\label{v-arbitrary-constant}
\end{eqnarray}
where coefficients $c_{\pm m}^\pm$ are uniquely defined from the initial condition $v(\cdot,0)=v_0 \in L^2_{\rm per}(0,L)$, and  $v_{\pm\lambda(k_m)}^\pm(x,t)$ are given by  \eqref{Fourier-basis-I}--\eqref{Fourier-basis-III} if $0<\pi m < L$,  by \eqref{Fourier 0}--\eqref{Fourier 0 gen} if $\pi m =L$, and by \eqref{Fourier-basis-1-new}--\eqref{Fourier-basis-2-new} if $\pi m > L$. 

\begin{remark}\label{r periodic}
When $L\notin \pi\mathbb N$, the functions $v_{\pm\lambda(k_m)}^\pm(x,t)$ in the decomposition \eqref{v-arbitrary-constant} are all obtained from the eigenfunctions associated to nonzero eigenvalues $\pm \lambda(k_m)$. When $L=\pi m$, the eigenvalues $\pm \lambda(k_m)$ vanish and the associated eigenfunctions only provide the two linearly independent solutions \eqref{Fourier 0}. The generalized eigenfunctions associated to the eigenvalue $\lambda(k_m)=0$ must be used in this case to obtain the other two solutions in \eqref{Fourier 0 gen}.
\end{remark}

\section{Akhmediev breather (AB)}
\label{s AB}

By using the Darboux transformation in Proposition~\ref{prop-Darboux},
we obtain AB from the constant solution $u = 1$. We describe 
the associate Lax spectrum in Section~\ref{ss AB Lax} and construct periodic solutions of the linearized NLS equation in Section~\ref{ss AB LNLS}.

Let $\lambda_0 \in (0,1)$ and define 
the particular solution $\varphi = (p_0,q_0)^T$ 
of the Lax system (\ref{lax-1})--(\ref{lax-2})
 with $u = 1$ and $\lambda = \lambda_0$:
\begin{equation}
\label{AB-eigen}
\left\{ \begin{array}{l} 
\displaystyle
p_0(x,t) = \sqrt{\lambda_0 - \frac{i}{2} k_0} \; e^{\frac{1}{2} (-ik_0 x + \sigma_0 t)} - \sqrt{\lambda_0 + \frac{i}{2} k_0} \; 
e^{\frac{1}{2} (ik_0 x - \sigma_0 t)}, \\
\displaystyle
q_0(x,t) = -\sqrt{\lambda_0 + \frac{i}{2} k_0} \; e^{\frac{1}{2} (-i k_0 x + \sigma_0 t)} + \sqrt{\lambda_0 - \frac{i}{2} k_0} \; 
e^{\frac{1}{2} (i k_0 x - \sigma_0 t)},
\end{array}
\right.
\end{equation}
where $k_0 = 2 \sqrt{1-\lambda_0^2} \,\in (0,2)$ and $\sigma_0 = \lambda_0 k_0$. 
Elementary computations give 
\begin{eqnarray*}
&&|p_0|^2 + |q_0|^2 = 4 \left[ \cosh(\sigma_0 t) - \lambda_0 \cos(k_0 x) \right], \\
&&|p_0|^2 - |q_0|^2 = 2 k_0 \sin(k_0 x), \\
&& p_0 \bar{q}_0 = 2 \cos(k_0 x) -2  \lambda_0 \cosh(\sigma_0 t) + i k_0 \sinh(\sigma_0 t). 
\end{eqnarray*}
and the one-fold Darboux transformation (\ref{DT-potential}) yields the formula for AB:
\begin{equation}
\label{AB-u}
\hat{u}_0(x,t) = - 1 + \frac{2(1-\lambda_0^2) \cosh(\sigma_0 t) + i \sigma_0 \sinh(\sigma_0 t)}{\cosh(\sigma_0 t) - \lambda_0 \cos(k_0 x)}.
\end{equation}
The AB solution $\hat u_0$ is $L$-periodic in $x$ with $L = 2\pi/k_0 \,>\pi$ and 
\[
\lim_{t \to \pm \infty} \hat{u}_0(x,t) = 1 - 2 \lambda_0^2 \pm i k_0 \lambda_0 
= 
\left(\sqrt{1-\lambda_0^2} \pm i \lambda_0 \right)^2,
\]
from which it follows that $\lim\limits_{t \to \pm \infty} |\hat{u}_0(x,t)| = 1$.
The complementary transformation (\ref{DT-squared}) gives 
\begin{equation*}
| \hat{u}_0(x,t)|^2 = 1 + \lambda_0 k_0^2 \frac{\cosh(\sigma_0 t) \cos(k_0 x) - \lambda_0}{\left[ \cosh(\sigma_0 t) - \lambda_0 \cos(k_0 x)\right]^2},
\end{equation*}
which is consistent with the exact solution (\ref{AB-u}).

\subsection{Lax spectrum at AB}
\label{ss AB Lax}

For the Lax system (\ref{lax-1})--(\ref{lax-2}), we consider both $L$-periodic and $L$-antiperiodic eigenfunctions $\varphi = \varphi(x,t)$ in $x$. We use the Darboux transformation (\ref{fund-matrix}) and the result of Lemma \ref{lemma-spectrum} to determine the Lax spectrum for AB, which is illustrated in  Figure~\ref{fig-Lax-2}.

\begin{figure}[htb]
	\begin{center}
		\includegraphics[width=8cm,height=6cm]{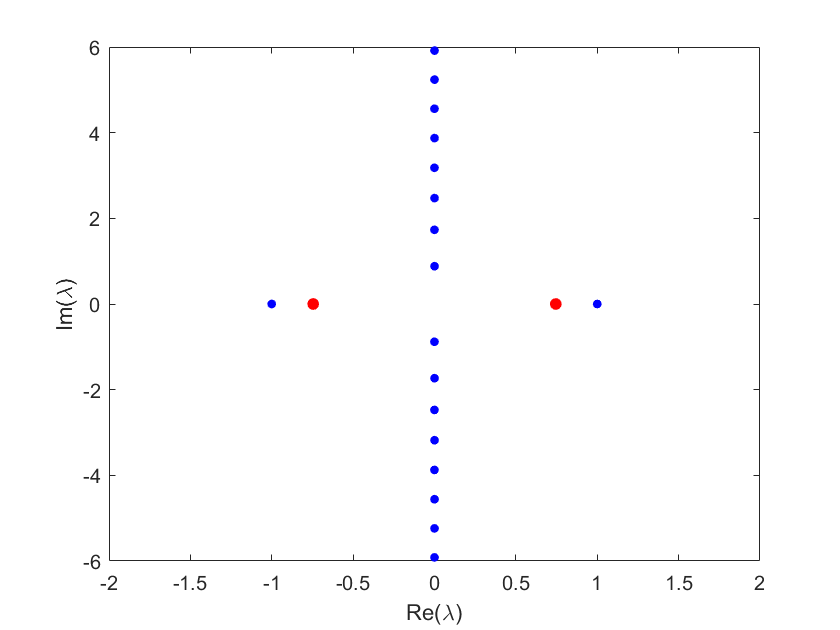}
	\end{center}
	\caption{The union  $\Sigma_{AB}^{(P)} \cup \Sigma_{AB}^{(A)}$ of the Lax spectra in $L^2_{\rm per}(0,L)$ and $L^2_{\rm antiper}(0,L)$ for AB. The red dots represent the eigenvalues $\{ +\lambda_0, -\lambda_0 \}$.}
	\label{fig-Lax-2}
\end{figure}

In the space of $L$-antiperiodic functions, we show that Lax spectrum of AB consists of the same eigenvalues (\ref{spectrum-constant-antiperiodic}) as for the constant-amplitude solution $u = 1$. The only difference between the two spectra is that the eigenvalues $\{ +\lambda_0, -\lambda_0 \}$ are geometrically double for $u = 1$, while they are geometrically simple and algebraically double for $u = \hat{u}_0$. 

\begin{lemma}\label{lax antiper}
  Consider AB given by (\ref{AB-u}) and assume $L \notin \pi \mathbb{N}_{\rm odd}$. The spectrum of the ZS spectral problem (\ref{lax-1}) with $u = \hat{u}_0$ 
  in $L^2_{\rm antiper}(0,L)$ consists of isolated eigenvalues
\[
 \Sigma_{\rm AB}^{(A)}= \{ \pm \lambda^{(A)}_m, \;\; m \in \mathbb{N}_{\rm odd}\}, \quad
\lambda^{(A)}_m := \sqrt{1-\frac{\pi^2}{L^2} m^2}, 
\]
with the following properties:
\begin{enumerate}
\item For each $m\in\mathbb{N}_{\rm odd}$, $m\not=1$, the eigenvalues $\pm\lambda^{(A)}_m$ are geometrically and algebraically double.

\item The eigenvalues $\lambda^{(A)}_1 = \lambda_0$ and $-\lambda^{(A)}_1 = -\lambda_0$ are geometrically simple and algebraically double with associated eigenfunctions $\varphi = (\hat{p}_0,\hat{q}_0)^T$ and $\varphi = (-\bar{\hat{q}}_0,\bar{\hat{p}}_0)^T$ and generalized eigenfunctions $\varphi_g = (\varphi_{1,1},\varphi_{1,2})^T$ and  $\varphi_g = (-\bar\varphi_{1,2},\bar\varphi_{1,1})^T$, where $\varphi_0 = (\hat{p}_0,\hat{q}_0)^T$ and $\varphi_1 =(\varphi_{1,1},\varphi_{1,2})^T$ are given by \eqref{norming-factor} and \eqref{gen-eigenvector} in Appendix \ref{a hatphi}.
\end{enumerate}
\end{lemma}

\begin{proof}
  The Darboux matrix $D(\lambda)$ given by (\ref{DT-matrix}) is $L$-periodic in $x$ and invertible for every $\lambda \neq \pm \lambda_0$. It follows 
  from the relation (\ref{fund-matrix}) that there is one-to-one correspondence between the $L$-antiperiodic solutions of the Lax systems with $u=1$ and $u= \hat{u}_0$ when  $\lambda \neq \pm \lambda_0$. Consequently, 
  with the exception of $m = 1$, the $L$-antiperiodic Lax spectrum for $u=\hat u_0$ is the same as the $L$-antiperiodic Lax spectrum for $u=1$ in (\ref{spectrum-constant-antiperiodic}) and the property (1) holds. 
The linearly independent eigenfunctions for the eigenvalues $\lambda=\pm \lambda^{(A)}_m$ are given in the form
  \begin{equation}
  \label{mode-continuous}
  \hat{\varphi} := \varphi + \frac{1}{\lambda - \lambda_0} \left[ \begin{array}{l} \hat{p}_0 \\ \hat{q}_0 \end{array} \right] \left[ -q_0 \;\; p_0 \right] \varphi, \quad 
  \hat{\phi} := \phi + \frac{1}{\lambda - \lambda_0} \left[ \begin{array}{l} \hat{p}_0 \\ \hat{q}_0 \end{array} \right] \left[ -q_0 \;\; p_0 \right] \phi,
  \end{equation}
where the two linearly independent eigenfunctions $\varphi$ and $\phi$ are
given by (\ref{eigenvectors-constant-background}) if $0<L<\pi m$ 
and by (\ref{eigenvectors-constant-background-gamma}) if $L>\pi m$. 
The marginal case $L=\pi m$ is excluded by the assumption.

For $\lambda=\lambda_0$, transformation (\ref{DT-eigen}) gives the eigenfunction $\varphi_0 = (\hat{p}_0,\hat{q}_0)^T$ of the Lax system with $u = \hat{u}_0$ and it is easy to check that $\varphi_0$ is $L$-antiperiodic in $x$. For $\lambda = -\lambda_0$ we have the eigenfunction $\varphi = (-\bar{\hat{q}}_0,\bar{\hat{p}}_0)^T$ due to the symmetry in Remark \ref{remark-symmetry}.  Hence $\{ +\lambda_0, -\lambda_0 \}$ belong to the  $L$-antiperiodic Lax spectrum for $u=\hat u_0$. It remains to show that $\lambda_0$ is geometrically simple and algebraically double, the result for $-\lambda_0$ following then by the symmetry of the Lax system.

For this part of the proof, we rely on the explicit computation of the expansion into Laurent series of the $2\times2$ matrix solution $\hat{\Phi}(\lambda)$ to the Lax system with $u=\hat u_0$ given in Appendix~\ref{a hatphi}.
The vector $\phi_0$ given by \eqref{non-periodic-vector} is a second linearly independent solution to  the Lax system (\ref{lax-1})--(\ref{lax-2}) for $u = \hat{u}_0$ and $\lambda = \lambda_0$. Since it is not $L$-antiperiodic in $x$, we deduce that $\lambda_0$ is geometrically simple. Next, $\varphi_1$ given by \eqref{gen-eigenvector} is $L$-antiperiodic and satifies $(\mathcal L-\lambda_0 I)\varphi_1=\varphi_0$, whereas
 $\varphi_2$ given by \eqref{second-generalized-eigenvector} satisfies $(\mathcal L-\lambda_0 I)\varphi_2=\varphi_1$, but it is not $L$-antiperiodic. This implies that $\lambda_0$ is algebraically double and completes the proof.
\end{proof}

\begin{remark}  
For an alternative proof that $\lambda_0$ is algebraically double, 
we can check the Fredholm condition for the eigenfunction $\varphi_0$ and the first generalized eigenfunction $\varphi_1$. Taking the eigenfunction $\varphi_0^* = (\bar{\hat{q}}_0,\bar{\hat{p}}_0)^T$ of the adjoint problem $\left( \mathcal{L}^* - \lambda_0 I \right) \varphi_0^* = 0$ and the inner product $\langle \cdot, \cdot \rangle$ in $L^2(0,L)$, we find that
\begin{eqnarray*}
	\langle \varphi_0^*, \varphi_0 \rangle = 
\lambda_0^2 \int_{0}^L \frac{\cosh(\sigma_0 t - ik_0 x) - \lambda_0}{(\cosh(\sigma_0 t) - \lambda_0 \cos(k_0 x))^2} dx = 0
\end{eqnarray*}
and
  \begin{eqnarray*}
	\langle \varphi_0^*, \varphi_1 \rangle 
	&=& \frac{\lambda_0}{2} \int_{0}^L \frac{\cos(k_0 x)}{\cosh(\sigma_0 t) - \lambda_0 \cos(k_0 x)} dx \\
	&& 
	+ \frac{2\lambda_0^2}{k_0^2} \int_{0}^L \frac{(\cosh(\sigma_0 t + i k_0 x) - \lambda_0) (\cosh(\sigma_0 t - i k_0 x) - \lambda_0)}{(\cosh(\sigma_0 t) - \lambda_0 \cos(k_0 x))^2} dx\\
	&=&
	\frac{2 \lambda_0^2}{k_0^2} \int_{0}^L \frac{\cosh^2(\sigma_0 t) + \lambda_0^2 \cos^2(k_0 x) - 2\lambda_0^2}{(\cosh(\sigma_0 t) - \lambda_0 \cos(k_0 x))^2} dx
	=\frac{2 \lambda_0^2}{k_0^2} L \not=0.
  \end{eqnarray*}
Since $\langle \varphi_0^*, \varphi_0 \rangle=0$, there exists the generalized eigenfunction $\varphi_1$ satisfying $(\mathcal{L} - \lambda_0 I) \varphi_1 = \varphi_0$ in $L^2_{\rm antiper}(0,L)$. Since $\langle \varphi_0^*, \varphi_1 \rangle \neq 0$, there is no the second generalized eigenfunction 
$\varphi_2$ satisfying $(\mathcal{L} - \lambda_0 I) \varphi_2 = \varphi_1$  in $L^2_{\rm antiper}(0,L)$. This implies that $\lambda_0$ is algebraically double. 
\end{remark}

For $L$-periodic solutions, we show that Lax spectrum of AB consists of the same eigenvalues (\ref{spectrum-constant-periodic}) as for the constant-amplitude solution $u = 1$. Moreover, algebraic multiplicitilies of the eigenvalues coincide.

\begin{lemma}\label{lax per}
  Consider AB given by (\ref{AB-u}) and assume $L \notin \pi \mathbb{N}_{\rm even}$. The spectrum of the ZS spectral problem (\ref{lax-1}) with $u = \hat{u}_0$ in $L^2_{\rm per}(0,L)$ consists of isolated eigenvalues
\[
 \Sigma_{\rm AB}^{(P)}= \{ \pm \lambda^{(P)}_m, \;\; m \in \{0,\mathbb{N}_{\rm even}\}\}, \quad
\lambda^{(P)}_m := \sqrt{1-\frac{\pi^2}{L^2} m^2}, 
\]
with the following properties:
\begin{enumerate}
\item For each $m\in\mathbb{N}_{\rm even}$, the eigenvalues $\pm\lambda^{(P)}_m$ are geometrically and algebraically double.
  \item The eigenvalues $\lambda^{(P)}_0 = 1$ and $-\lambda^{(P)}_0 = -1$ are algebraically simple with associated eigenfunctions $\varphi = (\hat{\varphi}_1,\hat{\varphi}_2)^T$ and $\varphi = (-\bar{\hat{\varphi}}_2,\bar{\hat{\varphi}}_1)^T$ respectively, where 
  $\hat{\varphi} = (\hat{\varphi}_1,\hat{\varphi}_2)^T$ is given~by 
\begin{equation}
\label{varphi-1}
\hat{\varphi} =
\left[ \begin{array}{c} 1\\ -1\end{array} \right] - 
\frac{p_0 + q_0}{1-\lambda_0} \left[ \begin{array}{c} \hat{p}_0 \\ \hat{q}_0   \end{array} \right].
\end{equation}
\end{enumerate}
\end{lemma}

\begin{proof}
As in the proof of Lemma~\ref{lax antiper}, the set of eigenvalues and their geometric and algebraic multiplicities are found from the Darboux matrix $D(\lambda)$ in (\ref{DT-matrix}) and the transformation (\ref{fund-matrix}). 
Since $D(\lambda)$ is $L$-periodic in $x$ and invertible for every $\lambda = \pm \lambda_m^{(P)}$, there is one-to-one correspondence between the $L$-periodic solutions of the Lax systems with $u = 1$ and $u = \hat{u}_0$. Moreover, the explicit expressions (\ref{mode-continuous}) for 
eigenfunctions $\hat{\varphi}$ and $\hat{\phi}$ hold for every $\lambda = \pm \lambda_m^{(P)}$.

For the eigenvalue $\lambda = \lambda^{(P)}_0 = 1$, only one linearly independent eigenfunction 
$\varphi = \hat{\varphi} = (\hat{\varphi}_1,\hat{\varphi}_2)^T$ in $L^2_{\rm per}(0,L)$ exists in the form (\ref{varphi-1}). In order to check the algebraic multiplicity of $\lambda=1$, we take eigenfunction $\varphi^* = (\bar{\hat{\varphi}}_2,\bar{\hat{\varphi}}_1)^T$ of the adjoint problem 
$(\mathcal{L}^* - I) \varphi^* = 0$ and compute the scalar product
\begin{eqnarray*}
	\langle \varphi^*, \varphi \rangle 
	&=& -2 \int_{0}^L \left[ 1 + \frac{2\lambda_0}{1-\lambda_0} 
	\frac{p_0 \bar{q}_0 + \bar{p}_0 q_0 + |p_0|^2 + |q_0|^2}{|p_0|^2 + |q_0|^2} + \frac{4 \lambda_0^2}{(1-\lambda_0)^2} \frac{p_0 q_0 (\bar{p}_0 + \bar{q}_0)^2}{(|p_0|^2 + |q_0|^2)^2} \right] dx \\
	&=& -2 \frac{1+\lambda_0}{1-\lambda_0} \int_{0}^L \left[ 1 + 2 \lambda_0 
	\frac{\cosh(\sigma_0 t) \cos(k_0 x) - \lambda_0 - i \lambda_0  \sinh(\sigma_0 t) \sin(k_0 x)}{\left[\cosh(\sigma_0 t) - \lambda_0 \cos(k_0 x)\right]^2} \right] dx \\
	&=& -2 \frac{1+\lambda_0}{1-\lambda_0} L.
\end{eqnarray*}
Since $\langle \varphi^*, \varphi \rangle \neq 0$, there exists no generalized eigenfunction satisfying  $(\mathcal{L} - I) \hat{\varphi}_g = \hat{\varphi}$ 
in $L^2_{\rm per}(0,L)$ so that the eigenvalue $\lambda = 1$ is algebraically simple. The result for $\lambda=-1$ is a consequence of the symmetry of the Lax system in Remark~\ref{remark-symmetry}.
\end{proof}

\subsection{Linearized NLS equation at AB}\label{ss AB LNLS}

As in the case of the constant-amplitude solution $u=1$, we construct $L$-periodic solutions of the linearized NLS equation at AB from the $L$-periodic and $L$-antiperiodic solutions of the Lax equations. These solutions are generated by the eigenvalues $\lambda^{(A)}_m$ and $\lambda^{(P)}_m$ in Lemmas~\ref{lax antiper} and~\ref{lax per}, respectively. 

We focus on the solutions related to the positive eigenvalues $\lambda^{(A)}_1 = \lambda_0$ and $\lambda^{(P)}_0 = 1$. By the symmetry of the Lax system, the negative eigenvalues $-\lambda^{(A)}_1 = -\lambda_0$ and $-\lambda^{(P)}_0 = -1$ provide the same solutions up to the sign change. The particular goal is to identify six linearly independent solutions of the linearized NLS equation at AB which  correspond to the six linearly independent solutions  in the decomposition (\ref{v-arbitrary-constant}) with $m=0$ and $m=1$ for the solutions of the linearized NLS equation at $u=1$. The correspondence is established by showing that the solutions constructed for AB become identical to the ones for $u=1$ asymptotically as $t\to\pm\infty$. 

The following theorem presents the main result of these computations. 

\begin{theorem}\label{six solutions}
	Consider AB given by (\ref{AB-u}). Solutions of the Lax system (\ref{lax-1})--(\ref{lax-2}) with $u=\hat u_0$ for the eigenvalues $\lambda= 1$ and $\lambda=\lambda_0$ generate the following six linearly independent $L$-periodic solutions of the linearized NLS equation (\ref{nls-lin}) at AB:
	\begin{enumerate}
		\item the solutions $v_1$ in (\ref{v1}) and $v_2$ in (\ref{v2}), which are asymptotically equivalent to the solutions $v_{\lambda(k_1)}^-$ and $\widetilde v_0^+$ respectively, in the decomposition (\ref{v-arbitrary-constant});
		\item the solution $w_2$ in (\ref{w2}), which is asymptotically equivalent to the solution $v_{-\lambda(k_1)}^-$ in the decomposition (\ref{v-arbitrary-constant});
		\item the solution $v$ in (\ref{element-basis}), which is asymptotically equivalent to the solution $\widetilde v_{0}^-$ in the decomposition (\ref{v-arbitrary-constant});
		\item the solutions $v$ in (\ref{new+}) and (\ref{new-}), which are asymptotically equivalent to the solutions $v_{-\lambda(k_1)}^+$ and $v_{\lambda(k_1)}^+$ respectively, in the decomposition (\ref{v-arbitrary-constant}).
	\end{enumerate}
\end{theorem}

The six solutions in this theorem are computed explicitly in the next three subsections.

\begin{remark}
Due to the two exponentially growing solutions in item (4) of Theorem \ref{six solutions}, AB is linearly unstable. This agrees with the main conclusion of \cite{GS-2021} based on symbolic computations.
For periods $L\in(\pi,2\pi)$ the eigenvalues $\lambda_0$ and $1$ are the only positive eigenvalues of the Lax system, see Figure \ref{fig-Lax-2}. For larger periods $L>2\pi$, there are additional positive eigenvalues which lead to  exponentially growing solutions for the linearized NLS equation at~AB.
  \end{remark}

\subsubsection{Solutions related to $\lambda = 1$}

Recall from Lemma \ref{lax per} that $\lambda=1$ is an algebraically simple eigenvalue in $\Sigma_{\rm AB}^{(P)}$ associated with  eigenfunction $\hat\varphi$ given by \eqref{varphi-1}.  The second linearly independent solution of the Lax system (\ref{lax-1})--(\ref{lax-2}) for $u = \hat{u}_0$ and $\lambda = 1$ 
is obtained from the second vector in (\ref{eigenvectors-constant-background-zero}) by using the transformation formula (\ref{fund-matrix})--(\ref{DT-matrix}) with $\lambda = 1$:
\begin{equation}
\label{varphi-2}
\hat\phi = \left[ \begin{array}{c} x+it+1 \\ -x-it\end{array} \right] - 
\frac{(x+it)(p_0 + q_0)+q_0}{1 - \lambda_0} \left[ \begin{array}{c} \hat{p}_0 \\ \hat{q}_0   \end{array} \right].
\end{equation}

By using the $L$-periodic eigenfunction $\hat\varphi$ in (\ref{varphi-1}) in 
Pair I in Table \ref{table-1}, we obtain 
the following two $L$-periodic solutions of the linearized NLS equation (\ref{nls-lin}):
\begin{equation}\label{v1}
v_1(x,t) = -\frac{2\lambda_0 (1+\lambda_0)}{1-\lambda_0} \; 
\frac{\sin(k_0 x) \left[ k_0 \cosh(\sigma_0 t) + 2 i \lambda_0 \sinh(\sigma_0 t)\right]}{\left[ \cosh(\sigma_0 t) - \lambda_0 \cos(k_0 x)\right]^2}.
\end{equation}
and
\begin{eqnarray}
\nonumber
 v_2(x,t) &=&
 \frac{2i (1+\lambda_0)}{1-\lambda_0} \; 
\left[ \frac{i k_0 \lambda_0 \sinh(\sigma_0 t) \cosh(\sigma_0 t)}{\left[ \cosh(\sigma_0 t) - \lambda_0 \cos(k_0 x)\right]^2} \right. \\
  && \left. + \frac{(1-2\lambda_0^2) \cosh^2(\sigma_0 t) - \lambda_0^2 \cos^2(k_0 x) +2 \lambda_0^2}{\left[ \cosh(\sigma_0 t) - \lambda_0 \cos(k_0 x)\right]^2} \right].
\label{v2}
\end{eqnarray}
As $t \to \pm \infty$, the periodic solution $v_1$ decays to $0$, whereas the periodic solution $v_2$ approaches  a nonzero constant. These two solutions 
are asymptotically equivalent to the solutions $v_{\lambda(k_1)}^-$ and $\widetilde v_0^+$ in the decomposition (\ref{v-arbitrary-constant}).

By using both the $L$-periodic eigenfunction $\hat\varphi$ in (\ref{varphi-1}) and the non-periodic solution $\hat\phi$ in (\ref{varphi-2}) in 
Pair II in Table \ref{table-1}, we obtain the following two 
non-periodic solutions of the linearized NLS equation (\ref{nls-lin}):
\begin{eqnarray}
\label{v3}
&& v_3(x,t) = x v_1(x,t) + t v_2(x,t) + f_1(x,t), \\
\label{v4}
&& v_4(x,t) = x v_2(x,t) - t v_1(x,t) + f_2(x,t), 
\end{eqnarray}
where the periodic parts $f_1$ and $f_2$ are given by 
\begin{eqnarray*}
f_1(x,t)& = & \frac{1+\lambda_0}{1-\lambda_0} \; 
\left[ 1 + \frac{4 i \lambda_0^2}{k_0} \sinh(\sigma_0 t) \frac{\lambda_0 \cosh(\sigma_0 t) - \cos(k_0 x)}{\left[ \cosh(\sigma_0 t) - \lambda_0 \cos(k_0 x) \right]^2} \right. \\
  &&
  \left. + 
  2 \lambda_0^2 \frac{\cosh^2(\sigma_0 t) - \cos^2(k_0 x) - i \sin(k_0 x) \sinh(\sigma_0 t)}{\left[ \cosh(\sigma_0 t) - \lambda_0 \cos(k_0 x)\right]^2}
 \right. \\
&& 
\left.  - \lambda_0 k_0 \frac{\cosh(\sigma_0 t) \sin(k_0x)}{\left[ \cosh(\sigma_0 t) - \lambda_0 \cos(k_0 x)\right]^2} \right].
\end{eqnarray*}
and
\begin{eqnarray*}
 f_2(x,t) &= & i \frac{1+\lambda_0}{1-\lambda_0} \; 
\left[ 1 - \frac{k_0 \lambda_0}{1-\lambda_0^2} \frac{\sin(k_0 x)}{\cosh(\sigma_0 t) - \lambda_0 \cos(k_0 x)}  + \frac{i \lambda_0 k_0  \cosh(\sigma_0 t) \sinh(\sigma_0 t)}{\left[ \cosh(\sigma_0 t) - \lambda_0 \cos(k_0 x)\right]^2} \right. \\
  && 
  \left. +
 \frac{4 \lambda_0 (1-\lambda_0) \cos(k_0 x)}{\cosh(\sigma_0 t) - \lambda_0 \cos(k_0 x)} -
\lambda_0^2 \frac{\cosh(2 \sigma_0 t) + \cos(2k_0 x)}{\left[ \cosh(\sigma_0 t) - \lambda_0 \cos(k_0 x)\right]^2} \right.\\
&&  
\left.  +
\frac{\lambda_0 (1+\lambda_0) (1-\lambda_0^2) \cos(k_0 x)^2}{\left[ \cosh(\sigma_0 t) - \lambda_0 \cos(k_0 x)\right]^2}
\frac{2 \lambda_0^3 \cosh(\sigma_0 t) \cos(k_0 x)}{\left[ \cosh(\sigma_0 t) - \lambda_0 \cos(k_0 x)\right]^2}\right].
\end{eqnarray*}
Both solutions grow linearly in $x$ and are not $L$-periodic.
As $t \to \pm \infty$, the non-periodic solution $v_3$ becomes asymptotically periodic, because $v_1$ decays to $0$, and could represent ${v}_0^-$ in the decomposition (\ref{v-arbitrary-constant}). However, one needs to cancel the polynomial term in $x$ by using a linear combination with other solutions of the linearized NLS equation (\ref{nls-lin}). 

Finally, by using the non-periodic solution (\ref{varphi-2}) in 
Pair III in Table \ref{table-1}, we obtain two other non-periodic solutions of the linearized NLS equation (\ref{nls-lin}), which are quadratic with respect to $x$. As is described in the recent symbolic computations in \cite{GS-2021}, 
such quadratic solutions in $x$ play no role in the proof of Theorem \ref{six solutions}.

\subsubsection{Solutions related to $\lambda = \lambda_0$}

By Lemma~\ref{lax antiper}, the eigenvalue $\lambda_0$ is geometrically simple with $L$-antiperiodic eigenfunction $\varphi_0 = (\hat{p}_0,\hat{q}_0)^T$ given by  (\ref{norming-factor}). A second linearly independent solution of the Lax system is the non-periodic solution $\phi_0$ given by (\ref{non-periodic-vector}) in Appendix~\ref{a hatphi}.

By using  Pair I in Table \ref{table-1} with $\varphi_0$, we obtain the following two $L$-periodic solutions of the linearized NLS equation (\ref{nls-lin}):
\begin{equation}
w_1 = \hat{p}_0^2 - \bar{\hat{q}}_0^2  = \frac{\lambda_0^2 \sin(k_0 x) \left[ 
	k_0 \cosh(\sigma_0 t) + 2 i \lambda_0 \sinh(\sigma_0 t)\right]}{2 \left[\cosh(\sigma_0 t) - \lambda_0 \cos(k_0 x)\right]^2} 
= -\lambda_0 k_0^{-2} \frac{\partial \hat{u}_0}{\partial x} 
\label{w1}
\end{equation}	
and
\begin{equation}
w_2 = i(\hat{p}_0^2 + \bar{\hat{q}}_0^2)  = \frac{\lambda_0^2 \left[ 
	k_0 \sinh(\sigma_0 t) \cos(k_0 x) + 2 i \lambda_0 \cosh(\sigma_0 t) \cos(k_0 x) - 2i \right]}{2 \left[\cosh(\sigma_0 t) - \lambda_0 \cos(k_0 x)\right]^2} = -k_0^{-2} \frac{\partial \hat{u}_0}{\partial t}.
\label{w2}
\end{equation}
These are neutral modes generated by the translational symmetries of the NLS equation (\ref{nls-u}) in $x$ and $t$. Note that $w_1$ is proportional to the solution $v_1$ in \eqref{v1},
\begin{equation}
\label{v1-w1}
v_1 = -\frac{4(1+\lambda_0)}{\lambda_0 (1-\lambda_0)} w_1.
\end{equation}
As $t \to \pm \infty$, the two periodic solutions $w_1$ and $w_2$ decay to $0$. These two solutions are asymptotically equivalent to the solutions 
$v_{\lambda(k_1)}^-$ and $v_{-\lambda(k_1)}^-$ in the decomposition (\ref{v-arbitrary-constant}). 

Next, we record the following algebraic computations:
\begin{eqnarray*}
&&-\bar{q}_0 p_+(\lambda_0) + p_0 \bar{q}_+(\lambda_0) = 4i \sin(k_0 x), \\
&&-\bar{q}_0 p_-(\lambda_0) + p_0 \bar{q}_-(\lambda_0) = 0, \\
&&-\bar{q}_0 p_+(\lambda_0) - p_0 \bar{q}_+(\lambda_0) = 4 \lambda_0 \sinh(\sigma_0 t) - 2 i k_0 \cosh(\sigma_0 t), \\ 
&&-\bar{q}_0 p_-(\lambda_0) - p_0 \bar{q}_-(\lambda_0) = 4 \left[ \lambda_0 \cosh(\sigma_0 t) - \cos(k_0x) \right] - 2i k_0 \sinh(\sigma_0 t).
\end{eqnarray*}
Then, by using Pair II in Table \ref{table-1} with the $L$-antiperiodic eigenfunction $\varphi_0$ in (\ref{norming-factor}) and the non-periodic solution $\phi_0$ in (\ref{non-periodic-vector}), we obtain 
the following two non-periodic solutions of the linearized NLS equation (\ref{nls-lin}):
	\begin{eqnarray}
	\label{w3}
	&& w_3(x,t) = -4 \lambda_0 x w_1(x,t) + 4(1-2\lambda_0^2) t w_2(x,t) + g_1(x,t), \\
	\label{w4}
	&& w_4(x,t) = -4 \lambda_0 x w_2(x,t) - 4(1-2\lambda_0^2) t w_1(x,t) + g_2(x,t), 
	\end{eqnarray}
	where the periodic parts $g_1$ and $g_2$ are given by 
	\begin{eqnarray*}
 g_1(x,t)& = & 
4 k_0^{-1} \left[ \cosh(\sigma_0 t) \sin(k_0 x) w_1(x,t) + \sinh(\sigma_0 t) \cos(k_0 x) w_2(x,t) \right] \\
	&&  + \lambda_0 
	\frac{2 \lambda_0 \cosh(\sigma_0 t) - 2 \cos(k_0x) - i k_0 \sinh(\sigma_0 t)}{\cosh(\sigma_0 t) - \lambda_0 \cos(k_0x)}
	\end{eqnarray*}
and 
	\begin{eqnarray*}
g_2(x,t) = 
4 k_0^{-1} \left[ \cosh(\sigma_0 t) \sin(k_0 x) w_2(x,t) - \sinh(\sigma_0 t) \cos(k_0 x) w_1(x,t) \right].
\end{eqnarray*}
Note that the components $g_1$ and $g_2$ are bounded in $t$ as $t \to \pm \infty$. In view of (\ref{v1-w1}), the linear combination 
\begin{eqnarray}
	\label{element-basis}
	v(x,t) = \lambda_0^2 \frac{1-\lambda_0}{1+\lambda_0} v_3(x,t) - w_3(x,t) = t s_0(x,t) + f_0(x,t),
\end{eqnarray}
where $s_0$ and $f_0$ are $L$-periodic in $x$ and bounded as $t \to \pm \infty$, e.g. 
\begin{eqnarray*}
s_0	&=& 2i \lambda_0^2 \left[ 1 - 2 \lambda_0^2 + \frac{ik_0 \lambda_0 \sinh(\sigma_0 t)}{\cosh(\sigma_0 t) - \lambda_0 \cos(k_0 x)} \right. \\
&& \left. + (1-\lambda_0^2) \frac{ik_0 \sinh(\sigma_0 t) \cos(k_0 x) + 2(1 - \lambda_0^2 \cos^2(k_0x))}{\left[ \cosh(\sigma_0 t) - \lambda_0 \cos(k_0x) \right]^2} \right].
\end{eqnarray*}
As $t \to \pm \infty$, the solution $v$ in (\ref{element-basis}) is asymptotically equivalent to the solution $\widetilde{v}_0^-$ in the decomposition  (\ref{v-arbitrary-constant}). 

\begin{remark}
	Since $v_2$ and $w_2$ are not linearly dependent from each other, there is no a linear combination of $v_4$ and $w_4$ which would be $L$-periodic in $x$.
\end{remark}

Finally, by using Pair III in Table \ref{table-1} with the non-periodic solution $\phi_0$ in (\ref{non-periodic-vector}), we obtain two non-periodic solutions which are quadratic in $x$. Again, such quadratic solutions in $x$ play no role in the proof of Theorem \ref{six solutions}.

\subsubsection{Solutions related to the generalized eigenfunction at $\lambda = \lambda_0$}

	With the account of $v_1$, $v_2$, $w_1$, $w_2$, the relation 
	(\ref{v1-w1}), and the linear combination (\ref{element-basis}), it remains to obtain two $L$-periodic solutions 
	of the linearized NLS equation (\ref{nls-lin}) at AB, 
	which would be asymptotically equivalent to the remaining solutions $v_{\lambda(k_1)}^+$ and $v_{-\lambda(k_1)}^+$ in the decomposition (\ref{v-arbitrary-constant}). These solutions will be constructed from  linear combinations of non-periodic solutions that grows linearly in $x$, just as the solution $v$ in (\ref{element-basis}).

By Lemma~\ref{lax antiper}, the eigenvalue $\lambda_0$ is algebraically double with the generalized eigenfunction $\varphi_1$ given by  
(\ref{gen-eigenvector}) in addition to the eigenfunction $\varphi_0$ given by (\ref{norming-factor}). It is natural to expect additional solutions for the linearized NLS equation to be obtained from the eigenfunction $\varphi_0$ and the generalized eigenfunction $\varphi_1$. It is surprising, however, that this is not the case. As a result, we have to use the eigenfunction $\varphi_0$ and the generalized eigenfunction $\varphi_1$ together with the non-periodic solutions $\phi_0$ and $\phi_1$ given by (\ref{non-periodic-vector}) 
and (\ref{second-non-periodic-vector}). 

By using the expansion of the $2\times2$ matrix solution $\hat{\Phi}(\lambda)$ computed in Appendix~\ref{a hatphi}, we write
\[
\hat\Phi(\lambda) = \left[ 2ik_0\varphi(\lambda), \phi(\lambda)\right],
\]
where
\begin{eqnarray}
  \label{varphilambda}
  \varphi(\lambda) &=&  \frac{\varphi_0}{\lambda - \lambda_0} 
+ \varphi_1 + \varphi_2 (\lambda - \lambda_0) + \mathcal{O}((\lambda - \lambda_0)^2) ,
\\
\label{philambda}
  \phi(\lambda) &=& \phi_0 + \phi_1 (\lambda - \lambda_0) + \mathcal{O}((\lambda - \lambda_0)^2).
\end{eqnarray}
Here $\varphi_0$, $\varphi_1$, and $\varphi_2$ are given by (\ref{norming-factor}), (\ref{gen-eigenvector}), and \eqref{second-generalized-eigenvector}, 
whereas $\phi_0$ and $\phi_1$ are given by (\ref{non-periodic-vector}) and \eqref{second-non-periodic-vector}. 

Both columns of $\hat\Phi(\lambda)$ being solutions of the Lax system (\ref{lax-1})--(\ref{lax-2}), the three pairs in Table~\ref{table-1} give solutions $v(\lambda)$ of the linearized NLS equation (\ref{nls-lin}) at AB. Expanding $v(\lambda)$ at $\lambda=\lambda_0$ generates a set of possible solutions to the linearized NLS equation (\ref{nls-lin}) at AB. It turns out that the $L$-periodic solutions and the linearly growing in $x$ solutions obtained from Pair I  in Table \ref{table-1} are all linear combinations of the previously obtained solutions and that the solutions obtained from Pair III in Table \ref{table-1} are all at least quadratic in $x$. As a result, the new suitable solutions to the linearized NLS equation (\ref{nls-lin}) must be obtained by using Pair II in Table \ref{table-1}.

Using Pair II in Table \ref{table-1} with the two columns in the matrix $\hat{\Phi}(\lambda)$ expanded at $\lambda=\lambda_0$, we obtain the following expansions
\begin{equation}
\label{expansion-v}
v = (2ik_0) \left[ \frac{w_{\pm}}{\lambda - \lambda_0} + v_{\pm} + \mathcal{O}(\lambda - \lambda_0) \right],
\end{equation}
where each term of the expansion gives a solution $v$ to the linearized NLS equation (\ref{nls-lin}) at AB. It follows that $w_+ = w_3$ and $w_- = w_4$ were previously obtained in (\ref{w3}) and (\ref{w4}) respectively. The next corrections in (\ref{expansion-v}) give two new solutions:
\begin{eqnarray*}
	v_+ &=& \varphi_{0,1} \phi_{1,1} - \bar{\varphi}_{0,2} \bar{\phi}_{1,2} + \varphi_{1,1} \phi_{0,1} - \bar{\varphi}_{1,2} \bar{\phi}_{0,2}, \\
	v_- &=& i \varphi_{0,1} \phi_{1,1} + i \bar{\varphi}_{0,2} \bar{\phi}_{1,2} + i \varphi_{1,1} \phi_{0,1} + i \bar{\varphi}_{1,2} \bar{\phi}_{0,2},
\end{eqnarray*}
where the first subscript stands for $\varphi_0$, $\varphi_1$, $\phi_0$, and $\phi_1$ and the second subscript stands for the first and second components of the $2$-vectors. For further computations of $v_\pm$ we obtain
\begin{eqnarray*}
p_0 p_+(\lambda_0) + \bar{q}_0 \bar{q}_+(\lambda_0) &= &2 \cos(k_0x) \left[ 2 \lambda_0 \sinh(\sigma_0 t) - i k_0 \cosh(\sigma_0 t) \right], \\
p_0 p_+(\lambda_0) - \bar{q}_0 \bar{q}_+(\lambda_0)& =& -2 \sin(k_0x) \left[ k_0 \sinh(\sigma_0 t) + 2 i \lambda_0 \cosh(\sigma_0 t) \right].
\end{eqnarray*}
After substitution of (\ref{norming-factor}), (\ref{gen-eigenvector}), 
(\ref{non-periodic-vector}), and (\ref{second-non-periodic-vector}) into $v_{\pm}$, we obtain
\begin{eqnarray*}
	v_{\pm}(x,t) = x r_{\pm}(x,t) + t s_{\pm}(x,t) + f_{\pm}(x,t),
\end{eqnarray*}
where the $L$-periodic parts are computed explicitly:
\begin{eqnarray*}
	r_+ &=& -\frac{8}{k_0^2} (3 - 2 \lambda_0^2) w_1, \\ 
	r_- &=& -\frac{8}{k_0^2} (1 - 4 \lambda_0^2) w_2 +
	\frac{2\lambda_0^2}{(1+\lambda_0)^2}v_2,
\end{eqnarray*}
\begin{eqnarray*}
	s_+ &=& \frac{4}{k_0} (1-2\lambda_0^2) (\hat{p}_0 p_+(\lambda_0) - \bar{\hat{q}}_0 \bar{q}_+(\lambda_0)) + \frac{16}{k_0^2} (1-2\lambda_0^2) \sinh(\sigma_0t) \sin(k_0 x)  w_1 \\	
	&&  +\frac{8}{k_0^2} (1-2\lambda_0^2) \left( 2 \cosh(\sigma_0t) \cos(k_0 x) - \lambda_0 \right) w_2 - 8 \lambda_0  w_2, \\[1ex]
	s_- &=& \frac{4i}{k_0} (1-2\lambda_0^2) (\hat{p}_0 p_+(\lambda_0) + \bar{\hat{q}}_0 \bar{q}_+(\lambda_0)) + \frac{16}{k_0^2} (1-2\lambda_0^2) \sinh(\sigma_0t) \sin(k_0 x)  w_2 \\	
	&&  -\frac{8}{k_0^2} (1-2\lambda_0^2) \left( 2 \cosh(\sigma_0t) \cos(k_0 x) - \lambda_0 \right) w_1 + 8 \lambda_0  w_1,
\end{eqnarray*}
and
\begin{eqnarray*}
	f_+ &=& \frac{1}{2i k_0} \left( p_0 p_+(\lambda_0) + \bar{q}_0 \bar{q}_+(\lambda_0)\right) + 
	\frac{i}{k_0} (\hat{p}_0 p_+(\lambda_0) - \bar{\hat{q}}_0 \bar{q}_+(\lambda_0)) \\
	&& + \frac{2}{k_0^2} \left( \cosh(\sigma_0 t) \cos(k_0 x) - \lambda_0 \right) (p_0 \hat{p}_0 - \bar{q}_0 \bar{\hat{q}}_0) - \frac{2i}{k_0^2} \sinh(\sigma_0 t) \sin(k_0 x)  (p_0 \hat{p}_0 + \bar{q}_0 \bar{\hat{q}}_0)\\	
	&& + \frac{2}{k_0^2} \sinh(\sigma_0 t) \cos(k_0 x) (\hat{p}_0 p_+(\lambda_0) - \bar{\hat{q}}_0 \bar{q}_+(\lambda_0))
	\\ && - \frac{2i}{k_0^2} \cosh(\sigma_0 t) \sin(k_0 x) (\hat{p}_0 p_+(\lambda_0) + \bar{\hat{q}}_0 \bar{q}_+(\lambda_0)) \\
	&&  + \frac{4}{k_0^3} 
	\cosh(2 \sigma_0t) \sin(2 k_0 x)  w_1  + \frac{4}{k_0^3} \sinh(2 \sigma_0t) \cos(2 k_0 x)  w_2, \\[1ex]
	f_- &=& \frac{1}{2 k_0} \left( p_0 p_+(\lambda_0) - \bar{q}_0 \bar{q}_+(\lambda_0)\right) -
	\frac{1}{k_0} (\hat{p}_0 p_+(\lambda_0) + \bar{\hat{q}}_0 \bar{q}_+(\lambda_0)) \\
	&& + \frac{2i}{k_0^2} \left( \cosh(\sigma_0 t) \cos(k_0 x) - \lambda_0 \right) (p_0 \hat{p}_0 + \bar{q}_0 \bar{\hat{q}}_0)
       + \frac{2}{k_0^2} \sinh(\sigma_0 t) \sin(k_0 x)  (p_0 \hat{p}_0 - \bar{q}_0 \bar{\hat{q}}_0)\\	
	&& + \frac{2i}{k_0^2} \sinh(\sigma_0 t) \cos(k_0 x) (\hat{p}_0 p_+(\lambda_0) + \bar{\hat{q}}_0 \bar{q}_+(\lambda_0))
	\\ && 	 + \frac{2}{k_0^2} \cosh(\sigma_0 t) \sin(k_0 x) (\hat{p}_0 p_+(\lambda_0) - \bar{\hat{q}}_0 \bar{q}_+(\lambda_0)) \\
	&&  + \frac{4}{k_0^3} \cosh(2 \sigma_0t) \sin(2 k_0 x)  w_2
         - \frac{4}{k_0^3} \sinh(2 \sigma_0t) \cos(2 k_0 x)  w_1.
\end{eqnarray*}
The $x$-growing part of $v_+$ is cancelled in the linear combination
\begin{eqnarray}
\nonumber
v(x,t) &=& k_0^2 v_+(x,t) - \frac{2 \lambda_0 (3-2\lambda_0^2)(1-\lambda_0)}{1+\lambda_0} v_3(x,t) \\
\label{new+}
&=&  t s_1(x,t) + k_0^2 f_+(x,t) - \frac{2 \lambda_0 (3-2\lambda_0^2)(1-\lambda_0)}{1+\lambda_0} f_1(x,t),
\end{eqnarray}
where
\begin{eqnarray*}
s_1(x,t) &=& 4i \lambda_0 (7-10\lambda_0^2)  
\frac{(2 \lambda_0^2-1) \cosh(\sigma_0 t)-i\lambda_0k_0 \sinh(\sigma_0 t) - \lambda_0 \cos(k_0 x)}{\cosh(\sigma_0 t) - \lambda_0 \cos(k_0 x)},
\end{eqnarray*}
and $f_+(x,t)$, $f_1(x,t)$ are $L$-periodic in $x$. 
Note that $t s_1(x,t)$ is irreducible in the sense that there are no other solutions of the linearized NLS equation (\ref{nls-lin}) with the same behavior as $t s_1(x,t)$. On the other hand, $s_1(x,t)$ and $f_1(x,t)$ are bounded as $t \to \pm \infty$, whereas $k_0^2 f_+(x,t)$ is unbounded. As $t \to \pm \infty$, we deduce that the exponentially growing component of $v(x,t)$ is given by 
\[
v(x,t) \sim -4 k_0^2 (1 - 4 \lambda_0^2) \cosh(\sigma_0 t) \cos(k_0 x) 
- 8 i \lambda_0 k_0 (3 - 4 \lambda_0^2) \sinh(\sigma_0 t) \cos(k_0x).
\]
We conclude that the solution $v$ in \eqref{new+} is $L$-periodic and asymptotically equivalent to the mode $v_{-\lambda(k_1)}^+$ in the decomposition (\ref{v-arbitrary-constant}) as $t \to \pm \infty$.

The $x$-growing part of $v_-$ is cancelled in the linear combination
\begin{eqnarray}
\nonumber
  v(x,t) &=& k_0^2 v_-(x,t) - \frac{2 (1-4\lambda_0^2)}{\lambda_0} w_4(x,t)
  -\frac{8\lambda_0^2 (1-\lambda_0)}{1+\lambda_0} v_4(x,t) \\
  \label{new-}
  &=& t s_2(x,t) +  k_0^2 f_-(x,t) 
  - \frac{2 (1-4\lambda_0^2)}{\lambda_0} g_2(x,t) -\frac{8\lambda_0^2 (1-\lambda_0)}{1+\lambda_0} f_2(x,t),
\end{eqnarray}
where
\begin{eqnarray*}
s_2(x,t) = -\lambda_0 k_0^2 (2\lambda_0^2+1)  
\frac{\sin(k_0 x)[k_0 \cosh(\sigma_0 t)+2i\lambda_0 \sinh(\sigma_0 t)]}{[\cosh(\sigma_0 t) - \lambda_0 \cos(k_0 x)]^2},
\end{eqnarray*}
and $f_-(x,t)$, $f_2(x,t)$, $g_2(x,t)$ are $L$-periodic in $x$. 
Again, the solution $v(x,t)$ grows exponentially in time as $t \to \pm \infty$ due to the unbounded component $k_0^2 f_-(x,t)$, according to 
$$
v(x,t) \sim -4 k_0^2 (1 - 4 \lambda_0^2) \sinh(\sigma_0 t) \sin(k_0 x) 
- 8 i \lambda_0 k_0 (3 - 4 \lambda_0^2) \cosh(\sigma_0 t) \sin(k_0x).
$$
We conclude that the solution $v$ in \eqref{new-} is $L$-periodic and asymptotically equivalent to the mode $v_{\lambda(k_1)}^+$ in the decomposition (\ref{v-arbitrary-constant}) as $t \to \pm \infty$. This completes the proof of the theorem.

\subsubsection{Solutions related to other eigenvalues}

We conclude this section with some comments on the solutions generated by the remaining eigenvalues in the Lax spectrum. These are the geometrically double eigenvalues $\{ \lambda_m^{(P)} \}_{m \in \mathbb{N}_{\rm even}}$ for $L$-periodic solutions and $\{ \lambda_m^{(A)}\}_{m \in \mathbb{N}_{\rm odd} \backslash \{1\}}$ for $L$-antiperiodic solutions. We exclude the case $L \in \pi\mathbb N$ when  $0$ is an eigenvalue of higher algebraic multiplicity and the two eigenfunctions alone were not enough to obtain the decomposition (\ref{v-arbitrary-constant}) for the constant solution~$u=1$. Then, using Pairs I and III in Table \ref{table-1}, from the associated eigenfunctions we obtain $L$-periodic solutions of the linearized NLS equation (\ref{nls-lin}) at AB which are asymptotically equivalent to the solutions $\{ v_{\pm \lambda(k_m)}^\pm \}_{m\in\mathbb N\setminus\{0,1\}}$, in the decomposition (\ref{v-arbitrary-constant}). Pair II in Table \ref{table-1} generates two $L$-periodic solutions which are linear combinations of the solutions $v_1$, $v_2$, and $w_2$ from Theorem~\ref{six solutions}. Together with the other three solutions from Theorem~\ref{six solutions}, the resulting set of solutions is asymptotically equivalent to the one in the  decomposition (\ref{v-arbitrary-constant}). While we do not attempt to prove completeness of this set, we refer to \cite[Section 4]{GS-2021} for a recent discussion of this question.

\section{Kuznetsov--Ma breather (KMB)} \label{s KMB}

Here we apply the same procedure of Section \ref{s AB} for KMB. Since KMB is localized in $x$, we have to consider the Lax spectrum and bounded solutions of the linearized NLS equation in the function space $L^2(\mathbb R)$. 

Let $\lambda_0 \in (1,\infty)$ and define 
the particular solution $\varphi = (p_0,q_0)^T$ of the Lax system (\ref{lax-1})--(\ref{lax-2}) with $u = 1$ and $\lambda = \lambda_0$:
\begin{equation}
\label{KMB-eigen}
\left\{ \begin{array}{l} 
\displaystyle
p_0(x,t) = \sqrt{\lambda_0 + \frac{1}{2} \beta_0} \; e^{\frac{1}{2} (\beta_0 x + i \alpha_0 t)} - \sqrt{\lambda_0 - \frac{1}{2} \beta_0} \; 
e^{-\frac{1}{2} (\beta_0 x + i \alpha_0 t)}, \\
\displaystyle
q_0(x,t) = -\sqrt{\lambda_0 - \frac{1}{2} \beta_0} \; e^{\frac{1}{2} (\beta_0 x + i \alpha_0 t)} + \sqrt{\lambda_0 + \frac{1}{2} \beta_0} \; 
e^{-\frac{1}{2} (\beta_0 x + i \alpha_0 t)},
\end{array}
\right.
\end{equation}
where $\beta_0 = 2 \sqrt{\lambda_0^2-1}$ and $\alpha_0 = \lambda_0 \beta_0$. 
Notice that $p_0$ and $q_0$ in (\ref{KMB-eigen}) are related symbolically to the ones for AB in \eqref{AB-eigen} through the equalities $k_0=i\beta_0$ and $\sigma_0=i\alpha_0$. Elementary computations give 
\begin{eqnarray} \nonumber 
&&|p_0|^2 + |q_0|^2 = 4 \left[ \lambda_0 \cosh(\beta_0 x) - \cos(\alpha_0 t) \right]\\  
\label{modKMB}
&&|p_0|^2 - |q_0|^2 = 2\beta_0\sinh(\beta_0 x)\\
\nonumber
&&p_0 \bar{q}_0 = -2 \cosh(\beta_0 x) + 2 \lambda_0 \cos(\alpha_0 t) + i \beta_0 \sin(\alpha_0 t),
\end{eqnarray}
so that the one-fold Darboux transformation (\ref{DT-potential}) yields the formula for KMB:
\begin{equation}
\label{KM-u}
\hat{u}_0(x,t) = - 1 + \frac{2(\lambda_0^2-1) \cos(\alpha_0 t) + i \alpha_0 \sin(\alpha_0 t)}{\lambda_0 \cosh(\beta_0 x) - \cos(\alpha_0 t)}.
\end{equation}
The complementary transformation (\ref{DT-squared}) gives a consistent relation 
\begin{equation*}
|\hat{u}_0(x,t)|^2 = 1 + \alpha_0 \beta_0 \frac{\lambda_0 - \cosh(\beta_0 x) \cos(\alpha_0 t)}{(\lambda_0 \cosh(\beta_0 x) - \cos(\alpha_0 t))^2},
\end{equation*}
which can also be derived from (\ref{KM-u}). KMB is periodic in $t$ with period $T = 2\pi/\alpha_0$ and localized in $x$ with $\lim\limits_{x \to \pm \infty} \hat{u}_0(x,t) =-1$.

\subsection{Lax spectrum at KMB} \label{ss KMB Lax}

As for AB, we use the Darboux matrix (\ref{DT-matrix}) to construct the bounded solutions of the Lax system (\ref{lax-1})--(\ref{lax-2}) with $u = \hat{u}_0$ from the bounded solutions of the Lax system with $u=1$ and then determine the Lax spectrum at KMB in~$L^2(\mathbb R)$. The Lax spectrum $\Sigma_{\rm KMB}$ is shown in Figure~\ref{fig-Lax-3} where the red dots show isolated eigenvalues $\{+\lambda_0,-\lambda_0\}$.

\begin{figure}[htb]
	\begin{center}
		\includegraphics[width=8cm,height=6cm]{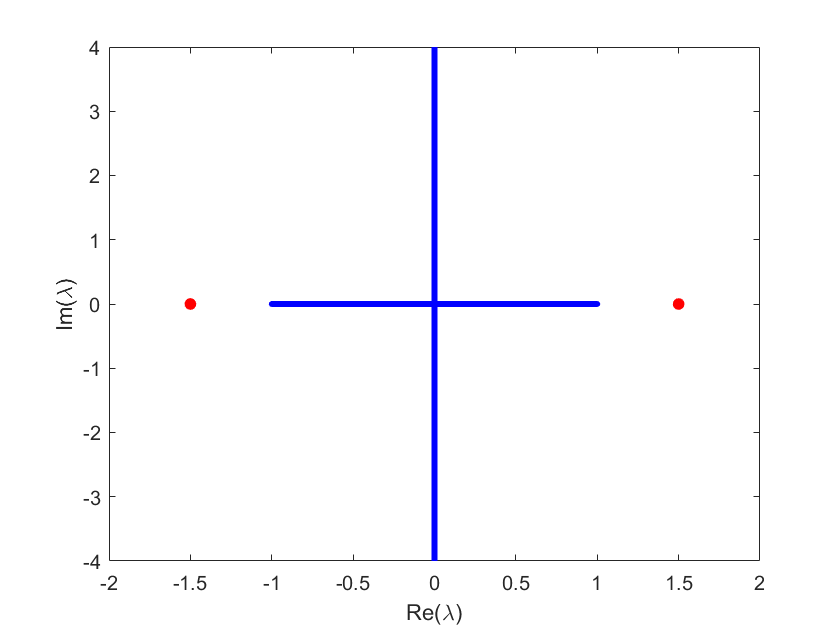}
	\end{center}
	\caption{The Lax spectrum $\Sigma_{\rm KMB}$ in $L^2(\mathbb{R})$ for KMB. }	
	\label{fig-Lax-3}
\end{figure}

The following lemma gives precisely the spectrum $\Sigma_{\rm KMB}$.

\begin{lemma}\label{LaxKMB}
Consider KMB given by (\ref{KM-u}). The spectrum of the ZS spectral problem (\ref{lax-1}) with $u = \hat{u}_0$ in $L^2(\mathbb R)$ is the set  
  \begin{equation}
\label{spectrum-KMB}
\Sigma_{\rm KMB} = i \mathbb{R} \cup [-1,1] \cup \{ \lambda_0, -\lambda_0\},
  \end{equation}
with the following properties:
\begin{enumerate}
\item For each $\lambda \in i \mathbb{R} \cup (-1,1)$, there exist two linearly independent bounded solutions.
\item For $\lambda = 1$ and $\lambda = -1$, there exists only one bounded solution.
\item  The eigenvalues $\lambda = \lambda_0$ and $\lambda = -\lambda_0$ are algebraically simple with associated eigenfunctions $\varphi = (\hat{p}_0,\hat{q}_0)^T$ and $\varphi = (\bar{\hat{q}}_0,-\bar{\hat{p}}_0)^T$ respectively.
  \end{enumerate}
\end{lemma}

\begin{proof}
The Darboux matrix $D(\lambda)$ given by (\ref{DT-matrix}) is invertible for $\lambda\not=\pm\lambda_0$. Moreover, both $D(\lambda)$ and its inverse are bounded in $x$ for $\lambda\not=\pm\lambda_0$. As a result, we have a one-to-one correspondence between the bounded solutions of the Lax systems with $u = 1$ and $u = \hat{u}_0$ for $\lambda\not=\pm\lambda_0$. This implies that, up to the values $\pm\lambda_0$, the ZS spectral problems
(\ref{lax-1}) with $u = 1$ and $u = \hat{u}_0$ have the same continuous spectrum $\Sigma_0$ given by \eqref{spectrum-constant}, so that properties (1) and (2) hold.

It remains to prove (3). Due to the  symmetry property of the Lax system, it is enough to show the result for $\lambda_0$.

The vector $\varphi = (p_0,q_0)^T$ given by \eqref{KMB-eigen} is a solution of the Lax system (\ref{lax-1})--(\ref{lax-2}) with $u = 1$ and $\lambda = \lambda_0$. The Darboux transformation (\ref{DT-eigen}) gives the solution $\varphi = (\hat{p}_0,\hat{q}_0)^T$ of the Lax system (\ref{lax-1})--(\ref{lax-2}) with $u = \hat{u}_0$ and $\lambda = \lambda_0$. From the formulas \eqref{KMB-eigen} and \eqref{modKMB} we find that $\hat{p}_0(x,t),\hat{q}_0(x,t) \to 0$ as $|x| \to \infty$ exponentially fast, hence $\varphi = (\hat{p}_0,\hat{q}_0)^T$ is an eigenfunction in $L^2(\mathbb{R})$ associated with $\lambda_0$. 
Furthermore, the Lax system (\ref{lax-1})--(\ref{lax-2}) having zero-trace, the Wronskian of any two solutions is constant both in $x$ and $t$. Since one solution is decaying to zero as $|x| \to \infty$, another linearly independent solution is necessarily growing at infinity. Consequently, $\lambda_0$ is geometrically simple.

For the algebraic multiplicity, we use the form $\left( \mathcal{L} - \lambda_0 I \right) \varphi_0 = 0$ with the eigenfunction $\varphi_0 = (\hat{p}_0,\hat{q}_0)^T \in L^2(\mathbb{R})$ and show that the linear nonhomogeneous equation 
\begin{equation}
\label{eigen-gen}
\left( \mathcal{L} - \lambda_0 I \right) \psi_0 = \varphi_0,
\end{equation}
does not have the generalized eigenfunction $\psi_0\in L^2(\mathbb R)$. The solvability condition for this equation is given by the Fredholm condition 
$\langle \varphi_0^*, \varphi_0 \rangle = 0$, where 
$\langle \cdot, \cdot \rangle$ is the inner product in $L^2(\mathbb{R})$ and $\varphi_0^*$ is an eigenfunction of the adjoint problem $\left( \mathcal{L}^* - \lambda_0 I \right) \varphi_0^* = 0$. A direct calculation shows that
$\varphi_0^* = (\bar{\hat{q}}_0,\bar{\hat{p}}_0)^T$, and then we compute
\begin{eqnarray*}
\langle \varphi_0^*, \varphi_0 \rangle &=& -\int_{\mathbb{R}} \frac{8 \lambda_0^2 p_0 q_0}{(|p_0|^2 + |q_0|^2)^2} dx \\
&=& \lambda_0^2 \int_{\mathbb{R}} \frac{\cosh(\beta_0 x + i \alpha_0 t) - \lambda_0}{(\lambda_0 \cosh(\beta_0 x) - \cos(\alpha_0 t))^2} dx \\
&=& -\lambda_0 \beta_0^{-1}  \frac{\lambda_0 \sinh(\beta_0 x) + i \sin(\alpha_0 t)}{\lambda_0 \cosh(\beta_0 x) - \cos(\alpha_0 t)} \biggr|_{x \to -\infty}^{x \to +\infty} =-2 \lambda_0 \beta_0^{-1}.
\end{eqnarray*}
Consequently, $\langle \varphi_0^*, \varphi_0 \rangle\not=0$ so that $\lambda_0$ is a simple eigenvalue and property (3) holds.
\end{proof}

\subsection{Linearized NLS equation at KMB} \label{ss KMB LNLS}

As in the case of the constant-amplitude solution $u=1$, we construct bounded solutions of the linearized NLS equation \eqref{nls-lin} at KMB from the bounded solutions of the Lax equations for $\lambda\in\Sigma_{\rm KMB}$ in Lemma~\ref{LaxKMB}. Recall that the solutions are unbounded in $x$ when $\lambda\notin\Sigma_{\rm KMB}$.

Here we focus on solutions generated by the eigenvalue $\lambda_0$, and in particular on those which are decaying to zero as $|x| \to \infty$. The eigenvalue $\lambda = -\lambda_0$ produces the same solutions of the linearized NLS equation, due to the symmetry in Remark \ref{remark-symmetry}. 
The following theorem provides the main result of these computations. 

\begin{theorem}\label{t kmb}
Consider KMB given by (\ref{KM-u}). The eigenvalue $\lambda_0$ of the Lax system generates three linearly independent exponentially decaying solutions of the linearized NLS equation (\ref{nls-lin}). These solutions are proportional to the three derivatives of KMB with respect to $x$, $t$, and $\lambda_0$.
\end{theorem}

\begin{remark}
Besides the three exponentially decaying solutions in Theorem \ref{t kmb}, we also find three linearly independent bounded solutions which are asymptotically constant as $x\to\pm\infty$. These six solutions are the analgue for KMB of the six solutions given in Theorem~\ref{six solutions} for~AB.
\end{remark}

\subsubsection{Solutions related to $\lambda= \lambda_0$}

Since $\lambda_0$ is a simple eigenvalue in $\Sigma_{\rm KMB}$, the eigenfunction $\varphi_0 = (\hat{p}_0,\hat{q}_0)^T$ provides an exponentially decaying solution of the Lax system (\ref{lax-1})--(\ref{lax-2}) with $\lambda=\lambda_0$. The second linearly independent solution is exponentially growing in $x$. According to Remark~\ref{hatphi kmb} in Appendix~\ref{a hatphi} this second solution is given by:
\begin{equation}
\label{growing-vector}
\phi_0 = \left[ \begin{array}{c} p_0 \\ q_0 \end{array} \right] + 
4 \left[-\lambda_0 x + i (1-2\lambda_0^2) t + \beta_0^{-1} \sinh(\beta_0 x + i \alpha_0 t) \right] \left[ \begin{array}{c} \hat{p}_0 \\ \hat{q}_0    \end{array} \right].
\end{equation}

By using Pair I in Table \ref{table-1} with $\varphi_0$, we obtain the solutions
\begin{eqnarray}
w_1(x,t) = \hat{p}_0^2 - \bar{\hat{q}}_0^2 = -\frac{\lambda_0^2 \sinh(\beta_0 x) (\beta_0 \cos(\alpha_0 t) + 2 i \lambda_0 \sin(\alpha_0 t))}{2 (\lambda_0 \cosh(\beta_0 x) - \cos(\alpha_0 t))^2} 
\label{mode-1}
\end{eqnarray}
and
\begin{eqnarray}
w_2(x,t) = i\hat{p}_0^2 +i\bar{\hat{q}}_0^2 =\frac{i\lambda_0^2 (2 \lambda_0 \cosh(\beta_0 x) \cos(\alpha_0 t) - 2 + i \beta_0 \cosh(\beta_0 x) \sin(\alpha_0 t))}{2 (\lambda_0 \cosh(\beta_0 x) - \cos(\alpha_0 t))^2},
\label{mode-2}
\end{eqnarray}
which are periodic in $t$ and exponentially decaying in $x$. It turns out that these solutions are proportional to the derivatives of $\hat{u}_0$ with respect to $x$ and $t$, 
\[
w_1= \lambda_0 \beta_0^{-2} \frac{\partial \hat{u}_0}{\partial x},\quad
w_2 = \beta_0^{-2} \frac{\partial \hat{u}_0}{\partial t}.
\]
Hence these solutions are  generated by the symmetries of the NLS equation (\ref{nls-u}) with respect to translation in $x$ and $t$.

While Pair III in Table \ref{table-1} with $\phi_0$ gives exponentially growing solutions, Pair II with $\varphi_0$ and $\phi_0$ gives two bounded solutions:
\begin{eqnarray}
\label{mode-3}
w_3(x,t) &=& -4 \lambda_0 x w_1(x,t) + 4 (1-2\lambda_0^2) t w_2(x,t) + f_1(x,t),\\
\label{mode-4}
w_4(x,t) &=& -4 \lambda_0 x w_2(x,t) - 4 (1-2\lambda_0^2) t w_1(x,t) + f_2(x,t),
\end{eqnarray}
where $w_1$ and $w_2$ are given by \eqref{mode-1} and \eqref{mode-2}, respectively, and
\begin{eqnarray*}
	f_1(x,t) &=& 2 \lambda_0 \cos(\alpha_0 t) \frac{2 \lambda_0 \cos(\alpha_0 t) - (1+\lambda_0^2) \cosh(\beta_0 x)}{\left[ \lambda_0 \cosh(\beta_0 x) - \cos(\alpha_0 t) \right]^2} \\
&& + 4 i \lambda_0 \beta_0^{-1} \sin(\alpha_0 t) \frac{(2\lambda_0^2 -1) \cos(\alpha_0 t)  - \lambda_0^3 \cosh(\beta_0 x)}{\left[ \lambda_0 \cosh(\beta_0 x) - \cos(\alpha_0 t) \right]^2},
\\
f_2(x,t)& =& 4i \lambda_0^2 \beta_0^{-1} \frac{\sinh(\beta_0 x)}{\lambda_0 \cosh(\beta_0 x) - \cos(\alpha_0 t)}.
\end{eqnarray*}
Here $f_1$ is exponentially decreasing as $|x| \to \infty$, whereas $f_2$ is bounded but not decaying as $|x| \to \infty$, and both $f_1$ and $f_2$ are periodic in $t$. Consequently, $w_3$ is also exponentially decaying in $x$, and a direct computation shows that it is proportional to the derivative of $\hat{u}_0$ with respect to $\lambda_0$,
\[
\frac{d \hat{u}_0}{d \lambda_0} = x \beta_0^{-1} \beta_0'(\lambda_0) \frac{\partial \hat{u}_0}{\partial x} + t \alpha_0^{-1} \alpha_0'(\lambda_0) \frac{\partial \hat{u}_0}{\partial t} + \frac{\partial \hat{u}_0}{\partial \lambda_0} = -\lambda_0^{-1} v_3.
\]
In this computation, $\lambda_0$ is an arbitrary parameter and we write $\beta_0 = \beta_0(\lambda_0)= 2 \sqrt{\lambda_0^2 - 1}$ and $\alpha_0=\alpha_0(\lambda_0) = \lambda_0 \beta_0(\lambda_0)$.
The solution $w_4$ is asymptotically constant, with
\begin{equation}\label{limv4}
\lim_{x \to \pm \infty} w_4(x,t) =\pm 4i \lambda_0 \beta_0^{-1}.
\end{equation}
The solutions $w_1$, $w_2$, and $w_3$ are the three linearly independent exponentially decaying solutions in Theorem~\ref{t kmb}.  

\subsubsection{Solutions related to $\lambda \in \Sigma_0$}

First, we consider the solutions of the linearized NLS equation which are asymptotically constant, but not decaying to $0$. These solutions are obtained from Pairs I and II in Table \ref{table-1} for $\lambda=1$ and from Pair II for any $\lambda \in \Sigma_0$. We are looking for a suitably chosen linear combination of these solutions with $w_4$ which might lead to a fourth exponentially decaying solution of the linearized NLS equation (\ref{nls-lin}). We show below that this is not the case.

For $\lambda=1$, the Lax system has the bounded solution $\hat\varphi$ in \eqref{varphi-1} and the unbounded solution $\hat\phi$ in \eqref{varphi-2} in which $(p_0,q_0)$ is given by (\ref{KMB-eigen}).
Using Pair~I of Table~\ref{table-1} with $\hat\varphi$ we find the solutions
\begin{equation}\label{e kmb v1}
  v_1(x,t)=  \frac{2\lambda_0(1+\lambda_0)}{1-\lambda_0}\frac{ \sinh(\beta_0 x) (\beta_0 \cos(\alpha_0 t) + 2 i \lambda_0 \sin(\alpha_0 t))}{[\lambda_0 \cosh(\beta_0 x) - \cos(\alpha_0 t)]^2} 
\end{equation}
and
\begin{eqnarray}
  v_2(x,t)&=&  \frac{2i}{(1-\lambda_0)^2}\left[\frac{-\lambda_0 \cosh(\beta_0 x) +(2\lambda_0^2-1) \cos(\alpha_0 t) + 2 i \lambda_0\beta_0 \sin(\alpha_0 t)}{ \lambda_0 \cosh(\beta_0 x) - \cos(\alpha_0 t)}\right. \nonumber \\
&&\left.+\lambda_0^2\frac{\beta_0^2 \sinh^2(\beta_0 x)+
(-2 \cosh(\beta_0 x) + 2 \lambda_0 \cos(\alpha_0 t) + i \beta_0 \sin(\alpha_0 t))^2}{4 \left[\lambda_0 \cosh(\beta_0 x) - \cos(\alpha_0 t)\right]^2}\right].
\label{e kmb v2}
\end{eqnarray}
The solution $v_1$ is proportional to $w_1$ given in \eqref{mode-1},
\[
v_1 = -\frac{4(1+\lambda_0)}{\lambda_0 (1-\lambda_0)} w_1,
\]
whereas the solution $v_2$ is asymptotically constant with
\[
\lim_{x \to \pm \infty} v_{2}(x,t) =-\frac{2i(1+\lambda_0)}{1-\lambda_0} .
\]
By using Pair~II with $\hat{\varphi}$ and $\hat{\phi}$, 
we find the bounded solution
\begin{equation}\label{e kmb v3}
  v_3(x,t)=  \left(x+\frac12\right)v_{1}(x,t) + t v_{2}(x,t) + g_1(x,t) ,
\end{equation}
where
\begin{eqnarray*}
g_1(x,t)= \frac1{(1-\lambda_0)^2}\left[1+
 \lambda_0^2\frac{\beta_0^2 \sinh^2(\beta_0 x)+
(-2 \cosh(\beta_0 x) + 2 \lambda_0 \cos(\alpha_0 t) + i \beta_0 \sin(\alpha_0 t))^2}{4 \left[\lambda_0 \cosh(\beta_0 x) - \cos(\alpha_0 t)\right]^2}\right] ,
\end{eqnarray*}
and a solution $v_4(x,t)$ which is unbounded in $x$. The solution $v_3$ grows linearly in $t$ and
\[
\lim_{x \to \pm \infty} v_{3}(x,t) = -t\,\frac{2i(1+\lambda_0)}{1-\lambda_0}+\frac{\lambda_0^2}{(1-\lambda_0)^2}.
\]
Pair~III gives two unbounded solutions. The three bounded, but not decaying to $0$, solutions $v_2$, $v_3$, and $w_4$ are linearly independent. Comparing their limits at $x=\pm\infty$ we conclude that there is no linear combination of these solutions which could lead to a localized solution.

Using Pair~II in Table \ref{table-1} with the two  linearly independent solutions of the Lax system for  $\lambda\in i\mathbb R \cup (-1,1)$ we do not find any new solutions. After some computations, we obtain that all these solutions are linear combinations of the exponentially decaying solutions $w_1$, $w_2$, obtained from the eigenvalue $\lambda=\lambda_0$, and the asymptotically constant solution $v_2$ obtained from $\lambda=1$.

The remaining solutions of the linearized NLS equation (\ref{nls-lin}) are obtained using Pairs I and III in Table \ref{table-1} for $\lambda\in i\mathbb R \cup (-1,1)$. Since the Darboux matrix $D(\lambda)$ in (\ref{DT-matrix}) is invertible with constant limits as $x\to\pm\infty$, these solutions are asymptotically the linear combinations of the solutions found for $u = 1$. For $\lambda\in i\mathbb R \cup (-1,1)$  the latter solutions are asymptotically periodic in $x$ with wavenumber $k=k(\lambda)=2\sqrt{1-\lambda^2}>0$. By analogy with the case $u=1$, we denote by  $\hat{v}_{\pm\lambda(k)}^\pm(x,t)$ the four solutions of the linearized NLS equation at KMB for $k \in (0,\infty) \backslash \{2\}$. Although only two linearly independent solutions are obtained for $k = 2$, the point $k = 2$ is of measure zero in the continuous spectrum $\Sigma_{\rm KMB}$.

\subsubsection{Localized solutions}

Based on these explicit computations, we expect that a solution $v \in C^0(\mathbb{R},L^2(\mathbb{R}))$ of the linearized NLS equation (\ref{nls-lin}) at KMB can be uniquely expressed in the linear superposition form 
\begin{eqnarray}
v(x,t)&=& c_1 v_1(x,t) + c_2 v_2(x,t) + c_3 v_3(x,t) \nonumber \\
&&+\int_0^\infty \left[\hat c_k^+  \hat v_{\lambda(k)}^+(x,t) + \hat c_k^-\hat v_{\lambda(k)}^-(x,t)  + \hat c_{-k}^+  \hat v_{-\lambda(k)}^+(x,t) + \hat c_{-k}^-\hat v_{-\lambda(k)}^-(x,t)\right] dk,\quad
\label{v-KMB}
\end{eqnarray}
where the coefficients $c_1$, $c_2$, $c_3$, and $\hat{c}_{\pm k}^\pm$ are uniquely determined by the coefficients by the initial condition $v(\cdot,0)=v_0\in L^2(\mathbb R)$. A rigorous justification of this formula requires an additional completeness proof which is outside the scope of this paper.

\begin{remark}
The decomposition (\ref{v-KMB}) precisely shows how many linearly independent solutions of the linearized NLS equation (\ref{nls-lin}) at KMB correspond to the point and continuous parts of the Lax spectrum $\Sigma_{\rm KMB}$. Interestingly, this decomposition is different from the complete set of solutions of the linearized NLS 
equation at the NLS soliton \cite{Kaup1,Kaup2} where perturbations to a single NLS soliton are decomposed over four exponentially decaying solutions which correspond to translations of the NLS soliton over four parameters and the four continuous families of eigenfunctions of the continuous spectrum. Here, we only found three exponentially decaying solutions. 
\end{remark}

\begin{remark}
	It follows from (\ref{v-KMB}) that the linear instability of KMB is related to the continuous spectrum $\Sigma_0$ in $\Sigma_{\rm KMB}$ with exactly the same growth rate as the one of the constant-amplitude background $u = 1$. This is in agreement with the numerical computation of unstable modes for KMB in \cite{Cuevas}, where KMB was truncated on a spatially periodic domain $[-L,L]$. According to Figs. 1,2, and 3 in \cite{Cuevas}, the number of unstable modes of KMB depends on the period $T$ for every fixed $L$. In the limit $T \to 0$ $(\lambda_0 \to \infty)$, the number of unstable modes corresponds to those of the constant-amplitude background $u = 1$. However, for each fixed $L$, the number of unstable modes decreases as $T$ decreases. Our analysis corresponds to the case $L = \infty$, when the unstable modes form a continuous spectrum which is independent of period~$T$. Indeed, the results in \cite{Cuevas} showed that the number of unstable modes increases when $L$ increases. 
\end{remark}

\section{Conclusion}
\label{sec-conclusion}

We have classified solutions of the linearized NLS equation (\ref{nls-lin}) 
at two breather solutions of the NLS equation (\ref{nls-u}) given by 
AB and KMB. In the case of AB, our results agree with the symbolic computations in \cite{GS-2021} where exponentially growing in time and spatially periodic 
solutions of the linearized NLS equation were discovered. In the case of KMB, 
we provide the set of solutions for characterizing the linear instability of breathers which was not achieved in the previous work \cite{Z} due to lack of spectral mapping properties. In both cases, the question of completeness was left opened and is the future open problem of highest priority.

Among further directions, it is worth mentioning that AB and KMB are particular solutions of the third-order Lax-Novikov equation 
\begin{equation}
\label{LN-3}
u''' + 6 |u|^2 u' + 2i c (u'' + 2|u|^2 u) + 4 b u' + 8i a u = 0,
\end{equation}
for $a = c = 0$. More general solutions of the third-order Lax--Novikov 
equation with $a = c = 0$ are represented by the double-periodic solutions which are periodic both in $x$ and $t$ \cite{Nail1,CPW}. Linear instabilities of the double-periodic solutions were recently explored in \cite{Pel-Double} by utilizying the Floquet theory both in $x$ and $t$. The linear unstable bands 
of the double-periodic solutions should correspond to the linear unstable 
modes of AB and KMB in the case of degeneration of the double-periodic solutions, this limiting procedure is still to be studied in future. 

Overall, characterizing instability of breathers on the constant-amplitude background is a more difficult problem compared to characterizing of the modulation instability of travelling periodic waves in the NLS equation 
\cite{CPW-2020,DecSegal}. Further understanding of the linear and nonlinear instability of breathers will provide better clarity of the formation of complex 
rogue wave patterns and integrable turbulence in the framework of the 
NLS equation (\ref{nls}).

\appendix

\section{Computation of $\hat\Phi(\lambda)$ near $\lambda = \lambda_0$}
\label{a hatphi}

Let $u = \hat u_0$ be AB given by \eqref{AB-u} and consider the $2\times2$ matrix solution ${\Phi}(\lambda)$ to the Lax system (\ref{lax-1})--(\ref{lax-2}) with $u=1$ in the form:
\begin{equation}
\label{fund-matrix-explicit}
\Phi(\lambda) = \left[ \begin{array}{cc} 
p_+(\lambda) & p_-(\lambda) \\
q_+(\lambda) & q_-(\lambda) \end{array}\right],
\end{equation}
where
\begin{eqnarray*}
	p_{\pm}(\lambda) &:=& \sqrt{\lambda - \frac{i}{2} k(\lambda)} \; e^{-\frac{1}{2} i k(\lambda) x + \frac{1}{2} \sigma(\lambda) t} \pm \sqrt{\lambda + \frac{i}{2} k(\lambda)} \; 
	e^{\frac{1}{2} i k(\lambda) x - \frac{1}{2} \sigma(\lambda) t}, \\
	q_{\pm}(\lambda) &:=& -\sqrt{\lambda + \frac{i}{2} k(\lambda)} \; e^{-\frac{1}{2} i k(\lambda) x + \frac{1}{2} \sigma(\lambda) t} \mp \sqrt{\lambda - \frac{i}{2} k(\lambda)} \; 
	e^{\frac{1}{2} i k(\lambda) x - \frac{1}{2} \sigma(\lambda) t},
\end{eqnarray*}
with $k(\lambda) := 2 \sqrt{1-\lambda^2}$, $\sigma(\lambda) := \lambda k(\lambda)$, and $\lambda \in \mathbb{R}$. 

\begin{remark}
The solution (\ref{AB-eigen}) used for the construction of $\hat u_0$ corresponds to the second column of $\Phi(\lambda)$ evaluated at $\lambda = \lambda_0$, so that $p_0 = p_-(\lambda_0)$ and $q_0 = q_-(\lambda_0)$.
\end{remark}

The $2\times2$ matrix solution $\hat{\Phi}(\lambda)$ of the Lax system (\ref{lax-1})--(\ref{lax-2}) with $u=\hat u_0$ is given by the transformation (\ref{fund-matrix}) and (\ref{DT-matrix}), or explicitly by 
\begin{equation}
\hat{\Phi}(\lambda) = \Phi(\lambda) + \frac{1}{\lambda - \lambda_0} \left[ \begin{array}{l} \hat{p}_0 \\ \hat{q}_0 \end{array} \right] \left[ -q_0 \;\; p_0 \right] \Phi(\lambda).
\end{equation}
Expanding $\hat{\Phi}(\lambda)$ into Laurent series near the simple pole at $\lambda = \lambda_0$, we write
\begin{equation}
\label{expansion}
\hat{\Phi}(\lambda) = \frac{\hat{\Phi}_{-1}}{\lambda - \lambda_0} 
+ \hat{\Phi}_0 + \hat{\Phi}_1 (\lambda - \lambda_0) + \mathcal{O}((\lambda - \lambda_0)^2),
\end{equation}
where
\begin{eqnarray*}
\hat{\Phi}_{-1} &:=& \left[ \begin{array}{l} \hat{p}_0 \\ \hat{q}_0 \end{array} \right] \left[ -q_0 \;\; p_0 \right] \Phi(\lambda_0), 	\\
\hat{\Phi}_0 &:=& \Phi(\lambda_0) +  \left[ \begin{array}{l} \hat{p}_0 \\ \hat{q}_0 \end{array} \right] \left[  -q_0 \;\; p_0  \right] \Phi'(\lambda_0), \\
\hat{\Phi}_1 &:=& \Phi'(\lambda_0) + \frac{1}{2} \left[ \begin{array}{l} \hat{p}_0 \\ \hat{q}_0 \end{array} \right] \left[  -q_0 \;\; p_0  \right] \Phi''(\lambda_0).
\end{eqnarray*}
Further expanding the eigenvalue problem
\[
(\mathcal L - \lambda I)\hat\Phi(\lambda) = 0,
\]
into Laurent series near the simple pole at $\lambda = \lambda_0$, we find successively at orders $\mathcal{O}((\lambda - \lambda_0)^{-1})$, $\mathcal O(1)$ and $\mathcal{O}(\lambda - \lambda_0)$ the equalities
\begin{eqnarray}
  \label{eigen pb}
 (\mathcal L - \lambda_0 I)\hat{\Phi}_{-1} = 0 , \quad
  (\mathcal L - \lambda_0 I)\hat{\Phi}_{0} = \hat{\Phi}_{-1} , \quad
  (\mathcal L - \lambda_0 I)\hat{\Phi}_{1} = \hat{\Phi}_{0}.
\end{eqnarray}
We use these equalities to identify the eigenfunctions and the generalized eigenfunctions of the Lax system for $\lambda=\lambda_0$ from the columns of $\hat{\Phi}_{-1}$, $\hat{\Phi}_{0}$, and $\hat{\Phi}_{1}$.

We record the following algebraic computations: 
\begin{eqnarray}
\label{technical-1}
\begin{array}{l} 
-q_0 p_+(\lambda_0) + p_0 q_+(\lambda_0) = 2i k_0, \\
-q_0 p_-(\lambda_0) + p_0 q_-(\lambda_0) = 0, \\
-q_0 p_+(\lambda_0) - p_0 q_+(\lambda_0) = 4 \sinh(\sigma_0 t - i k_0 x), \\ 
-q_0 p_-(\lambda_0) - p_0 q_-(\lambda_0) = 4 \cosh(\sigma_0 t - i k_0 x) - 4 \lambda_0.\end{array}
\end{eqnarray}
By using the first two lines of (\ref{technical-1}),  we obtain
\begin{equation}\label{hatphi-1}
  \hat{\Phi}_{-1} = [2ik_0 \varphi_0,0],
\end{equation}
  where 
\begin{equation}
\label{norming-factor}
\varphi_0 = \left[ \begin{array}{c} \hat{p}_0 \\ \hat{q}_0 \end{array} \right] 
= \frac{\lambda_0}{2 \left[ \cosh(\sigma_0 t) - \lambda_0 \cos(k_0 x) \right] } \left[ \begin{array}{c} -\bar{q}_0 \\ \bar{p}_0 \end{array} \right], 
\end{equation}
is precisely the $L$-antiperiodic eigenfunction of the Lax system (\ref{lax-1})--(\ref{lax-2}) for $u = \hat{u}_0$ and $\lambda = \lambda_0$.

Next, using the equalities
$$
\frac{d}{d \lambda} \sqrt{ \lambda \pm \frac{i}{2} k(\lambda)} = \mp \frac{i}{k} \sqrt{\lambda \pm \frac{i}{2} k(\lambda)}
$$
we compute the derivatives
\begin{eqnarray}
	\label{der-p-q}
	\begin{array}{l}
	p_{\pm}'(\lambda) = k(\lambda)^{-1} \left[ i + 2 i \lambda x + 2 (1-2\lambda^2)t \right] p_{\mp}(\lambda), \\
	q_{\pm}'(\lambda) = k(\lambda)^{-1} \left[ -i + 2 i \lambda x + 2 (1-2\lambda^2)t \right] q_{\mp}(\lambda). \end{array}
\end{eqnarray}
which together with the four equalities of (\ref{technical-1}), lead to
\begin{eqnarray*}
\left[ -q_0 \;\; p_0 \right] \Phi'(\lambda_0) 
= 4ik_0^{-1} \left[ \cosh(\sigma_0 t - i k_0 x) - \lambda_0, k_0( i \lambda_0 x + (1-2\lambda_0^2) t) + \sinh(\sigma_0 t - i k_0 x) \right]. 
\end{eqnarray*}
Then we compute
\begin{equation}\label{hatphi0}
  \hat{\Phi}_{0} = [2ik_0 \varphi_1,\phi_0],
\end{equation}
where
\begin{eqnarray}
\label{gen-eigenvector}
\varphi_1 = \frac{1}{2ik_0} \left[ \begin{array}{c} p_+(\lambda_0) \\
q_+(\lambda_0) \end{array} \right] + 
\frac{2}{k_0^2} \left[ \cosh(\sigma_0 t - i k_0 x) - \lambda_0 \right] \left[ \begin{array}{c} \hat{p}_0 \\ \hat{q}_0   \end{array} \right].
\end{eqnarray}
is the $L$-antiperiodic generalized eigenfunction satisfying $(\mathcal L-\lambda_0 I)\varphi_1=\varphi_0$, and
\begin{equation}
\label{non-periodic-vector}
\phi_0 = \left[ \begin{array}{c} p_0 \\ q_0 \end{array} \right] + 
4 \left[-\lambda_0 x + i (1-2\lambda_0^2) t + i k_0^{-1} \sinh(\sigma_0 t - i k_0 x) \right] \left[ \begin{array}{c} \hat{p}_0 \\ \hat{q}_0    \end{array} \right],	
\end{equation}
is a second linearly independent solution to the Lax system (\ref{lax-1})--(\ref{lax-2}) for $u = \hat{u}_0$ and $\lambda = \lambda_0$, which is not $L$-antiperiodic in $x$. 

Finally, by differentiating (\ref{der-p-q}) in $\lambda$ and using (\ref{technical-1}) we obtain   
\begin{eqnarray*}
-q_0 p_+''(\lambda_0) + p_0 q_+''(\lambda_0) &=& 
2i k_0^{-1} \left[ 4(i \lambda_0 x + (1-2\lambda_0^2) t)^2 - 1 \right]  \\
&& + 16 i k_0^{-2} \left[ i \lambda_0 x + (1-2\lambda_0^2) t \right] 
\sinh(\sigma_0 t - i k_0 x)  \\
&& + 16 i \lambda_0  k_0^{-3} 
\left[ \cosh(\sigma_0 t - i k_0 x) - \lambda_0 \right].
\end{eqnarray*}
and 
\begin{eqnarray*}
	-q_0 p_-''(\lambda_0) + p_0 q_-''(\lambda_0) &=& 
	16 i k_0^{-2} \left[ i \lambda_0 x + (1-2 \lambda_0^2) t \right] \cosh(\sigma_0 t - i k_0 x) \\
	&& + 16 i \lambda_0 k_0^{-3} \sinh(\sigma_0 t - i k_0 x) + 4i (ix - 4 \lambda_0 t).
\end{eqnarray*}
Then we compute
\begin{equation}\label{hatphi1}
  \hat{\Phi}_{1} = [2ik_0 \varphi_2,\phi_1],
\end{equation}
which consists of  non-$L$-antiperiodic functions, 
\begin{eqnarray}
  \nonumber
  \varphi_2 & =& \frac{1}{k_0^2} (\lambda_0 x - i (1-2\lambda_0^2)t) \left[ \begin{array}{c} p_0 \\ q_0 \end{array} \right] + 
\frac{1}{2k_0^2} \left[ \begin{array}{c} p_0 \\
-q_0 \end{array} \right]  \\
\nonumber
&&  
+ \frac{1}{2 k_0^2}\left[ 4(i \lambda_0 x + (1-2\lambda_0^2)t)^2 - 1 \right] \left[ \begin{array}{c} \hat{p}_0 \\ \hat{q}_0   \end{array} \right] \\
		\nonumber
	&& 
+\frac{4}{k_0^3} (i \lambda_0 x + (1-2\lambda_0^2)t) \sinh(\sigma_0 t - i k_0 x) \left[ \begin{array}{c} \hat{p}_0 \\ \hat{q}_0   \end{array} \right] \\
\label{second-generalized-eigenvector}
&&  
+ 
\frac{4 \lambda_0}{k_0^4} \left[ \cosh(\sigma_0 t - i k_0 x) - \lambda_0 \right] \left[ \begin{array}{c} \hat{p}_0 \\ \hat{q}_0   \end{array} \right]. 
\end{eqnarray}
satisfying $(\mathcal L-\lambda_0 I)\varphi_2=\varphi_1$, and
\begin{eqnarray}
\nonumber
	 \phi_1 & =& \frac{2}{k_0} (i\lambda_0 x + (1-2\lambda_0^2)t) \left[ \begin{array}{c} p_+(\lambda_0) \\ q_+(\lambda_0) \end{array} \right] + 
	\frac{i}{k_0} \left[ \begin{array}{c} p_+(\lambda_0) \\
		-q_+(\lambda_0) \end{array} \right] \\
		\nonumber
	&& 
	+\frac{8i}{k_0^2} (i \lambda_0 x + (1-2\lambda_0^2)t) \cosh(\sigma_0 t - i k_0 x) \left[ \begin{array}{c} \hat{p}_0 \\ \hat{q}_0   \end{array} \right] + 
	\frac{8i \lambda_0}{k_0^3} \sinh(\sigma_0 t - i k_0 x)  \left[ \begin{array}{c} \hat{p}_0 \\ \hat{q}_0   \end{array} \right] \\
	&&   
	+ 2i (ix - 4 \lambda_0 t) \left[ \begin{array}{c} \hat{p}_0 \\ \hat{q}_0   \end{array} \right].
	\label{second-non-periodic-vector}
\end{eqnarray}
satisfying $(\mathcal L-\lambda_0 I)\phi_1=\phi_0$.

\begin{remark}\label{hatphi kmb}
These computations can be transferred to KMB given by (\ref{KM-u}) by taking $k_0=i\beta_0$ and $\sigma_0=i\alpha_0$. In particular, we obtain two linearly independent solutions of the Lax system (\ref{lax-1})--(\ref{lax-2}) with $u = \hat{u}_0$ and $\lambda=\lambda_0$ for KMB from (\ref{norming-factor}) and (\ref{non-periodic-vector}). The first one is the solution $\varphi_0 = (\hat{p}_0,\hat{q}_0)^T$ in Lemma~\ref{LaxKMB} and the second one is the solution $\phi_0$ given by (\ref{growing-vector}).
\end{remark}

\begin{remark}
Derivatives of the matrix solution $\hat\Phi(\lambda)$ in $\lambda$ at $\lambda = \pm \lambda_0$  were computed in the recent work \cite{GS-2021} without discussing geometric and algebraic multiplicities of the eigenvalue~$\lambda_0$.
\end{remark}


\begin{thebibliography}{99}
	
\bibitem{AKNS} M.J. Ablowitz, D.J. Kaup, A.C. Newell, and H. Segur,
  ``The inverse scattering transform--Fourier analysis for nonlinear problems",
  Stud. Appl. Math. {\bf 53} (1974), 249--315.

\bibitem{Nail1} N.N. Akhmediev, V.M. Eleonskii, and N.E. Kulagin,
  ``Exact first-order solutions
  of the nonlinear Schr\"{o}dinger equation",
  Theor. Math. Phys. {\bf 72} (1987), 809--818.

\bibitem{AFM-2} M.A. Alejo, L. Fanelli, and C. Munoz,
  ``The Akhmediev breather is unstable",
  Sao Paulo J. Math. Sci. {\bf 13} (2019), 391--401. 

\bibitem{AFM} M.A. Alejo, L. Fanelli, and C. Munoz,
  ``Stability and instability of breathers in the $U(1)$ Sasa--Satsuma and nonlinear Schr\"{o}dinger models",
 Nonlinearity {\bf 34} (2021), 3429--3484.
  
\bibitem{Alejo} M.A. Alejo, L. Fanelli, and C. Munoz,
  ``Review on the stability 
  of the Peregrine and related breathers",
  Front. Phys. {\bf 8} (2020), 591995 (8 pages).

\bibitem{Bil2} D. Bilman and P.D. Miller,
  ``A robust inverse scattering transform for the focusing nonlinear Schr\"{o}dinger equation",
  Comm. Pure Appl. Math. {\bf 72} (2019), 1722--1805.

\bibitem{BK} G. Biondini and G. Kovacic,
  ``Inverse scattering transform for the focusing nonlinear Schr\"{o}odinger equation with nonzero boundary conditions",
  J. Math. Phys. {\bf 55} (2014), 031506 (22 pages).

\bibitem{CalSch2012} A. Calini and  C.M. Schober, ``Dynamical criteria for rogue waves in nonlinear Schrödinger models", Nonlinearity {\bf 25} (2012) R99--R116.

\bibitem{CalSch2019} A. Calini, C.M. Schober, and M. Strawn,
``Linear instability of the Peregrine breather: Numerical and analytical investigations", Appl. Numer. Math. {\bf 141} (2019), 36--43.

\bibitem{CPW} J. Chen, D.E. Pelinovsky, and R.E. White,
  ``Rogue waves on the double-periodic background in the focusing nonlinear Schr\"{o}dinger equation",
  Physical Review E {\bf 100} (2019), 052219 (18 pages).

\bibitem{CPW-2020} J. Chen, D.E. Pelinovsky, and R.E. White,
  ``Periodic standing waves in the focusing nonlinear Schrodinger equation: Rogue waves and modulation instability",
  Physica D {\bf 405} (2020), 132378 (13 pages).

\bibitem{ContPel} A. Contreras and D.E. Pelinovsky, ``Stability of multi-solitons in the cubic NLS equation", J. Hyper. PDEs {\bf 11} (2014), 329--353.

\bibitem{Kibler1} F. Copie, S. Randoux, and P. Suret, ``The physics 
of the one-dimensional nonlinear Schr\"{o}dinger equation in fiber optics: 
Rogue waves, modulation instability and self-focusing phenomena", 
Reviews in Physics {\bf 5} (2020), 100037 (17 pages)

\bibitem{Cuevas} J. Cuevas--Maraver, P.G. Kevrekidis, D.J. Frantzeskakis, 
N.I. Karachalios, M. Haragus, and G. James, ``Floquet analysis of Kuznetsov--Ma breathers: A path towards spectral stability of rogue waves", 
Phys. Rev. E {\bf 96} (2017), 012202 (8 pages).

\bibitem{DecSegal} B. Deconinck and B. L. Segal, ``The stability spectrum for elliptic solutions to the focusing NLS equation", Physica D {\bf 346} (2017) 1--19

\bibitem{Kibler2} J.M. Dudley, G. Genty, A. Mussot, A. Chabchoub, and F. Dias, ``Rogue waves and analogies in optics
and oceanography", Nature Reviews {\bf 1} (2019), 675.

\bibitem{NLS2} G. Fibich, {\em The Nonlinear Schr\"{o}dinger Equation: Singular Solutions and Optical Collapse}, Springer-Verlag (New York, 2015).

\bibitem{Garn} J. Garnier and K. Kalimeris,
  ``Inverse scattering perturbation theory for the nonlinear Schr\"{o}dinger equation with non-vanishing
  background",
  J. Phys. A: Math. Theor. {\bf 45} (2012), 035202 (13pp).

\bibitem{GS-2021} P.G. Grinevich and P.M. Santini,
  ``The linear and nonlinear instability of the Akhmediev breather",
Nonlinearity {\bf 34} (2021), 8331--8358.

\bibitem{Kaup1} D.J. Kaup,
  ``A perturbation expansion for the Zakharov-Shabat inverse scattering transform'',
  SIAM J. Appl. Math. {\bf 31} (1976), 121--133.

\bibitem{Kaup2} D.J. Kaup,
  ``Closure of the squared Zakharov-Shabat eigenstates'',
  J. Math. Anal. Appl. {\bf 54} (1976), 849--864.
  
\bibitem{Rogue1} C. Kharif, E. Pelinovsky, and A. Slunyaev, {\em Rogue Waves in the Ocean} (Springer, Heidelberg, 2009).

\bibitem{KPR} M. Klaus, D.E. Pelinovsky, and V.M. Rothos,
  ``Evans function for Lax operators with algebraically decaying potentials",
  J. Nonlin. Sci. {\bf 16} (2006), 1--44.

\bibitem{KH} C. Klein and M. Haragus,
  ``Numerical study of the stability of the Peregrine breather",
  Ann. Math. Sci. Appl. {\bf 2} (2017), 217--239.


\bibitem{Kuznetsov} E.A. Kuznetsov, ``Solitons in a parametrically unstable plasma,” Sov. Phys. Dokl. {\bf 22} (1977), 507--508.


\bibitem{Ma} Y.-C. Ma, ``The perturbed plane-wave solutions of the cubic Schr\"{o}dinger equation,” Stud. Appl. Math. {\bf 60} (1979), 43--58.

\bibitem{Pel-Double} D.E. Pelinovsky, ``Instability of double-periodic solutions in the nonlinear Schrodinger equation", Frontiers in Physics {\bf 9} (2021) 599146 (10 pages)

\bibitem{Peregrine} D.H. Peregrine, ``Water waves, nonlinear Schr\"{o}dinger equations and their solutions",	J. Austral. Math. Soc. B. {\bf 25} (1983), 16--43.

 \bibitem{NLS1} C. Sulem and P.L. Sulem,
\newblock  {\it The Nonlinear Schr{\"o}dinger Equation},
Springer-Verlag (New York, 1999).

\bibitem{Kevr-Salerno} J. Sullivan, E.G. Charalampidis, J. Cuevas-Maraver, P.G. Kevrekidis, and N.I. Karachalios, ``Kuznetsov–Ma breather-like solutions in the Salerno model", Eur. Phys. J. Plus {\bf 135} (2020) 607 (12 pages).


\bibitem{TW} M. Tajiri and Y. Watanabe, ``Breather solutions to the focusing nonlinear Schr\"{o}dinger equation",  Phys. Rev. E {\bf 57} (1998), 3510--3519.
  
\bibitem{Rogue2} S. Wabnitz (editor), {\em Nonlinear Guided Wave Optics: A Testbed for Extreme Waves}  (Iop Publishing Ltd, Bristol, 2018)


\bibitem{ZakGelash}  V.E. Zakharov and A.A. Gelash,  ``Nonlinear stage of modulation instability",  Phys. Rev. Lett. {\bf 111} (2013), 054101 (5 pages).

\bibitem{ZS} V.E. Zakharov and A.B. Shabat,
  ``Exact theory of two-dimensional
self-focusing and one-dimensional self-modulation of waves in nonlinear media",
Soviet Physics JETP {\bf 34} (1972), 62--69.

\bibitem{Z} J. Zweck, Y. Latushkin, J.L. Marzuola, and C.K.R.T. Jones,
  ``The essential spectrum of periodically stationary solutions of the complex Ginzburg-Landau equation'',
  J. Evol. Equ. {\bf 21} (2021), 3313--3329.


\end{thebibliography}
\end{document}